\relax
\documentclass[letterpaper]{article} 
\usepackage{aaai22}  
\usepackage{times}  
\usepackage{helvet}  
\usepackage{courier}  
\usepackage[hyphens]{url}  
\usepackage{graphicx} 
\usepackage{natbib}  
\usepackage{caption} 
\DeclareCaptionStyle{ruled}{labelfont=normalfont,labelsep=colon,strut=off} 
\usepackage[ruled,vlined,linesnumbered]{algorithm2e} 
\usepackage{algpseudocode} 
\usepackage[utf8]{inputenc} 
\usepackage[T1]{fontenc}    
\usepackage{url}            
\usepackage{booktabs}       
\usepackage{amsfonts}       
\usepackage{nicefrac}       
\usepackage{microtype}      

\usepackage{epsfig,epsf,fancybox}
\usepackage{amsmath}
\usepackage{mathrsfs}
\usepackage{amssymb}
\usepackage{amsfonts}
\usepackage{graphicx}
\usepackage{color} 
\usepackage{boxedminipage}
\usepackage{stmaryrd}
\usepackage{multirow}
\usepackage{booktabs}
\usepackage{cite}
\usepackage{float}
\usepackage{xcolor}
\usepackage{caption} 
\usepackage{subcaption} 

\usepackage{comment}

\newtheorem{definition}{Definition} %
\newtheorem{lemma}{Lemma} %
\newtheorem{theorem}{Theorem} %
\newtheorem{corollary}{Corollary} %
\newtheorem{remark}{Remark} 
\newenvironment{proof}{\paragraph{Proof:}}{\hfill$\square$} 

\title{Zeroth-order Optimization for Composite Problems with Functional Constraints}

\author{Zichong Li$^1$, Pin-Yu Chen$^{2}$, Sijia Liu$^{3}$, Songtao Lu$^{2}$, Yangyang Xu$^{1}$ 
}
\affiliations{$^1$Rensselaer Polytechnic Institute, $^2$IBM Research, $^3$Michigan State University\\
	\{liz19, xuy21\}@rpi.edu, \{Pin-Yu.Chen, songtao\}@ibm.com, liusiji5@msu.edu
}

\usepackage{mathtools}

\usepackage[normalem]{ulem} 

\newcommand{\vb}{{\mathbf{b}}}
\newcommand{\vc}{{\mathbf{c}}}

\newcommand{\ve}{{\mathbf{e}}}

\newcommand{\vh}{{\mathbf{h}}}

\newcommand{\vu}{{\mathbf{u}}}
\newcommand{\vv}{{\mathbf{v}}}
\newcommand{\vw}{{\mathbf{w}}}
\newcommand{\vx}{{\mathbf{x}}}
\newcommand{\vy}{{\mathbf{y}}}
\newcommand{\vz}{{\mathbf{z}}}

\newcommand{\vA}{{\mathbf{A}}}

\newcommand{\vH}{{\mathbf{H}}}
\newcommand{\vI}{{\mathbf{I}}}

\newcommand{\vQ}{{\mathbf{Q}}}
\newcommand{\vR}{{\mathbf{R}}}

\newcommand{\cL}{{\mathcal{L}}}

\newcommand{\cN}{{\mathcal{N}}}

\newcommand{\cS}{{\mathcal{S}}}

\newcommand{\cU}{{\mathcal{U}}}

\newcommand{\cX}{{\mathcal{X}}}

\newcommand{\vareps}{\varepsilon}

\newcommand{\EE}{\mathbb{E}} 
\newcommand{\RR}{\mathbb{R}} 
\newcommand{\ZZ}{\mathbb{Z}} 
\newcommand{\vzero}{\mathbf{0}} 
\newcommand{\vone}{{\mathbf{1}}} 

\newcommand{\dist}{\mathrm{dist}}    
\newcommand{\dom}{{\mathrm{dom}}} 
\newcommand{\tr}{{\mathrm{tr}}} 

\newcommand{\st}{\mbox{ s.t. }}

\DeclareMathOperator*{\argmin}{arg\,min} 

\DeclareMathOperator*{\Min}{minimize}

\newcommand{\bc}{\begin{center}}
\newcommand{\ec}{\end{center}}

\newcommand{\bdm}{\begin{displaymath}}
\newcommand{\edm}{\end{displaymath}}

\newcommand{\beq}{\begin{equation}}
\newcommand{\eeq}{\end{equation}}

\newcommand{\bfl}{\begin{flushleft}}
\newcommand{\efl}{\end{flushleft}}

\newcommand{\bt}{\begin{tabbing}}
\newcommand{\et}{\end{tabbing}}

\newcommand{\beqn}{\begin{eqnarray}}
\newcommand{\eeqn}{\end{eqnarray}}

\newcommand{\beqs}{\begin{align*}} 
\newcommand{\eeqs}{\end{align*}}  

\newtheorem{assumption}{Assumption}

\begin{document}
	
	\maketitle
	
	\begin{abstract}
		In many real-world problems, first-order (FO) derivative evaluations are too expensive or even inaccessible. 
		For solving these problems, zeroth-order (ZO) methods that only need function evaluations are often more efficient than FO methods or sometimes the only options. In this paper, we propose a novel zeroth-order inexact augmented Lagrangian method (ZO-iALM) to solve black-box optimization problems, which involve a composite (i.e., smooth+nonsmooth) objective and functional constraints. Under a certain regularity condition (also assumed by several existing works on FO methods), the query complexity of our ZO-iALM is $\tilde{O}(d\vareps^{-3})$ to find an $\vareps$-KKT point for problems with a nonconvex objective and nonconvex constraints, and $\tilde{O}(d\vareps^{-2.5})$ for nonconvex problems with convex constraints, where $d$ is the variable dimension. 
		This appears to be the first work that develops an iALM-based ZO method for functional constrained optimization and meanwhile achieves 
		query complexity results matching the best-known FO complexity results up to a factor of $d$. With an extensive experimental study, we show the effectiveness of our method. 
		The applications of our method span from classical optimization problems to practical machine learning examples such as resource allocation in sensor networks and adversarial example generation.
	\end{abstract}
	
	\section{Introduction}
	In many practical optimization problems such as black-box attack \cite{chen2017zoo}, we only have access to zeroth-order (ZO) function values but 
	no access to first-order (FO) or higher order derivatives \cite{liu2020primer}. 
	In this paper, 
	we consider \textit{nonconvex} problems with \textit{possibly nonconvex} constraints:
	\begin{equation}\label{eq:ncp} 
	f_0^*:=\min_{\vx\in\RR^d} \big\{f_0(\vx) := g(\vx)+h(\vx), \st \vc(\vx) = \vzero \big\},
	\end{equation}
	where $g$ is smooth but possibly nonconvex, $\vc=(c_1, \ldots, c_l):\RR^d\to\RR^l$ is a vector function with continuously differentiable components, and $h$ is closed convex but possibly nonsmooth and has a coordinate structure, i.e., $h(\vx)=\sum_{i=1}^d h(x_i)$. This formulation follows from \cite{li2021rate}, except that in this paper, only the function evaluations of $g$ and $\vc$, but not their gradients, are accessible. Such a formulation includes a large class of nonlinear constrained problems. 
	We remark that an inequality constraint $t(\vx)\le 0$ can be equivalently formulated as an equality constraint $t(\vx)+s=0$ by enforcing the nonnegativity of $s$, with equivalent stationarity conditions (c.f. \cite{li2021rate}). Note in \eqref{eq:ncp}, the inclusion of a coordinate-separable constraint set $\cX$ is equivalent to the addition of $I_{\cX}$ to the nonsmooth term $h$ in the objective $f_0$, where $I_{\cX}(\vx)=0$ if $\vx \in \cX$ and $+\infty$ otherwise.
	
	\subsection{Contributions}
	Our contributions are three-fold. 
	First, we design a zeroth-order accelerated proximal coordinate update (ZO-APCU) method for solving coordinate-structured strongly-convex composite (i.e., smooth+nonsmooth) problems. ZO-APCU appears to be the first PCU method \emph{with acceleration} by just using function values. It can be viewed as a ZO variant of the APCG in \cite{lin2014accelerated}. 
	To solve black-box optimization in the form of \eqref{eq:ncp}, we propose
	a novel zeroth-order inexact augmented Lagrangian method (ZO-iALM), by using ZO-APCU to design a zeroth-order inexact proximal point method (ZO-iPPM) to approximately solve each ALM subproblem. 
	Though any ZO method can be used as a subroutine in the iALM, the use of ZO-iPPM with the developed ZO-APCU is crucial to yield best-known query complexity results and also good numerical performance, as we demonstrate in the experiments. 
	
	Second, query complexity analysis is conducted on the proposed methods. We show that ZO-APCU needs ${O}(d\sqrt{\kappa}\log\frac{1}{\vareps})$ queries to produce an $\vareps$-stationary point of a $d$-dimensional strongly-convex composite problem with a condition number $\kappa$. 
	The ZO-iPPM has an 
	$\tilde{O}(d \vareps^{-2})$ complexity to give an $\vareps$-stationary point of a nonconvex composite problem. On solving \eqref{eq:ncp} that satisfies a regularity condition, the ZO-iALM has an $\tilde{O}(d\vareps^{-3})$ overall complexity to produce an $\vareps$-KKT point, and the complexity can be reduced to $\tilde{O}(d\vareps^{-\frac{5}{2}})$ if the constraints are affine. All our complexity results are (near) optimal or the best known. 
	To the best of our knowledge, complexity of ZO methods on nonconvex functional constrained problem \eqref{eq:ncp} has not been studied in the literature, thus our $\tilde{O}(d\vareps^{-3})$ result is completely new.
	
	Third, we use a coordinate gradient estimator while implementing the core solver ZO-APCU. To be able to yield high-accuracy solutions, we give a multi-point coordinate-wise gradient estimator and analyze its error bound. 
	Under the $j$-th order smoothness assumption for some $j \in \ZZ^+$, we show that the error of a $\max\{2, 2(j-1)\}$-point coordinate-wise gradient estimator is upper bounded by $O(a^j)$, where $a$ is the sampling radius. 
	This result is meaningful and important to yield high accuracy, because in practice $a$ cannot be too small, otherwise round-off errors will dominate. 
	
	Overall, we conduct a comprehensive study on ZO methods on solving nonconvex functional constrained black-box optimization, from multiple perspectives including complexity analysis, gradient estimator, and significantly improved performance on practical machine learning tasks and classical optimization problems.
	\subsection{Related Works} 
	
	
	In this subsection, we review previous works on the inexact augmented Lagrangian methods (iALMs) in the usual FO settings and the zeroth-order methods (ZOMs). 
	
	The iALM is one of the most common methods for solving constrained problems. It alternatingly updates the primal variable by inexactly minimizing the augmented Lagrangian (AL) function and the Lagrangian multiplier by dual gradient ascent. If the multiplier is fixed to zero, then iALM reduces to a standard penalty method, which usually has a worse practical performance than iALM. Previous works on iALM assume that FO derivatives of the objective function can be evaluated, therefore use FOMs to inexactly minimize the AL function in each primal update. 
	
	For convex nonlinear constrained problems, the iALM in \cite{doi:10.1287/ijoo.2021.0052}
	and the proximal-iALM in \cite{li2021inexact} can produce an $\vareps$-KKT point with $O(\vareps^{-1}|\log\vareps|)$ gradient evaluations, and the AL-based FOMs in \cite{xu2021iteration, xu2021first, ouyang2015accelerated, li2021inexact, nedelcu2014computational} can produce an $\vareps$-optimal solution with $O(\vareps^{-1})$ complexity. For strongly-convex problems, the complexity results can be respectively reduced to $O(\vareps^{-\frac{1}{2}}|\log\vareps|)$ for an $\vareps$-KKT point and $O(\vareps^{-\frac{1}{2}})$ for an $\vareps$-optimal solution, e.g., \cite{doi:10.1287/ijoo.2021.0052, li2021inexact, xu2021iteration, nedelcu2014computational, necoara2014rate}. 
	
	For nonconvex problems with nonlinear convex constraints, when Slater's condition holds, $\tilde O(\vareps^{-\frac{5}{2}})$ complexity results have been obtained by the AL or quadratic-penalty based FOMs in \cite{doi:10.1287/ijoo.2021.0052, lin2019inexact-pp} and the proximal ALM in \cite{meloiteration2020}. When a regularity condition (see Assumption \ref{assumption:regularity} below) holds, 
	the ALM in \cite{li2021rate} achieves $\tilde O(\vareps^{-\frac{5}{2}})$ complexity for nonconvex problems with nonlinear convex constraints and $\tilde O(\vareps^{-3})$ complexity for problems with nonconvex constraints.
	
	When gradients of the objective function are unavailable, ZOMs are the only tools available. Previous ZO works mainly focus on problems \textit{without} nonlinear functional constraints. Many existing ZOMs are modified from some gradient descent type FOMs, replacing the exact gradient $ \nabla f(\vx) $ by some gradient estimator $\tilde{\nabla} f(\vx)$. In section~\ref{sec:multi_pt}, we briefly review some existing gradient estimation frameworks including random search and finite difference. A more detailed overview can be found in \cite{liu2020primer} for ZOMs.
	
	
	\subsection{Notations, Definitions, and Assumptions}
	We use $\Vert \cdot \Vert$ for the Euclidean norm of a vector and the spectral norm of a matrix. 
	$[n]$ denotes the set $\{1,\ldots,n\}$. We use $\tilde{O}$ to suppress all log terms of $\vareps$ from the big-$O$ notation. We denote $J_c(\vx)$ as the Jacobian of $\vc$ at $\vx$. 
	The distance between a vector $\vx$ and a set $\cX$ is denoted as $\dist(\vx, \cX) = \min_{\vy \in \cX} \Vert \vx-\vy \Vert$. For a function $f$, we use $\partial f$ to denote the subdifferetial of $f$. For a differentiable function $f$, we use $\tilde\nabla f$ as an estimator of the gradient $\nabla f$. 
	The AL function of \eqref{eq:ncp} is
	\begin{equation}\label{eq:AL}
	\textstyle \cL_{\beta}(\vx,\vy) = f_0(\vx) + \vy^\top \vc(\vx) + \frac{\beta}{2} \Vert \vc(\vx) \Vert^2,
	\end{equation} 
	where $\beta > 0$ is the penalty parameter, and $\vy \in \RR^l$ is the multiplier or the dual variable. 
	
	\begin{definition}[$\vareps$-KKT point]\label{def:eps-kkt}
		A point $\vx \in \RR^d$ is an $\vareps$-KKT point of \eqref{eq:ncp} if there is $\vy\in\RR^l$ such that
		\begin{equation}\label{eq:kkt}
		\Vert \vc(\vx) \Vert \leq \vareps,\quad
		\dist\left(\vzero, \partial f_0(\vx)+J_c^\top(\vx) \  \vy \right) \leq \vareps.
		\end{equation}
	\end{definition}

	\begin{definition}[$k$-smoothness]
		For some $k \ge 1$, we say $f$ is $M_k$ $k$-smooth, if the $k$-th derivative of $f$ is $M_k$ Lipschitz continuous. 
	\end{definition}
	\begin{remark}
		Letting $k = 1$ above corresponds to the standard smoothness assumption.
	\end{remark}
	
	\begin{definition}[coordinate $k$-smooth]
		For some $k \ge 1$, we say $f$ is $M_k$ coordinate $k$-smooth, if the partial function $F_i(\vx_i) := f(\vx_{<i}, \vx_i, \vx_{>i})$ is $M_k$ $k$-smooth, $\forall\, i \in [d]$, where $\vx_{<i} := (x_1, \dots, x_{i-1})$ and $\vx_{>i} := (x_{i+1}, \dots, x_{d})$ are fixed.
	\end{definition}
	\begin{remark}
		If $f$ is $M_k$ $k$-smooth, then it must be $M^c_k$ coordinate $k$-smooth with some $M^c_k \leq M_k$.
	\end{remark}
	
	\begin{definition}[$\rho$-weakly convex]
		A function $f$ is $\rho$-weakly convex if $f+\frac{\rho}{2}\Vert \cdot \Vert^2$ is convex.
	\end{definition}
	\begin{remark}
		A function that is $L$-smooth is also $L$-weakly convex. However, its weak convexity constant can be much smaller than its smoothness constant.
	\end{remark}
	
	Throughout this paper, we make the following assumptions. 
	\begin{assumption}[smoothness and weak convexity]\label{assump:smooth-wc}
		In \eqref{eq:ncp}, $g$ is $L_0$-smooth and $\rho_0$-weakly convex. For each $j \in [l]$, $c_j$ is $L_j$-smooth and $\rho_j$-weakly convex.
	\end{assumption}
	
	\begin{assumption}[bounded domain]\label{assump:composite}
		In \eqref{eq:ncp}, $h$ is closed convex with a compact domain, i.e.,
		\begin{subequations}
			\begin{equation}\label{eq:bounded}
			D := \max_{\vx,\vx' \in \dom(h)} \Vert \vx-\vx' \Vert < \infty,
			\end{equation}
			\begin{equation}\label{eq:bounded_coor}
			D_i := \max_{\substack{\vx,\vx' \in \dom(h) \\ \vx_{[d] \backslash i} = \vx'_{[d] \backslash i} }} \Vert \vx-\vx' \Vert \le D, \forall i \in [d],
			\end{equation}
		\end{subequations}
		where $\vx_{[d] \backslash i} = (\vx_1, \dots, \vx_{i-1}, \vx_{i+1}, \dots, \vx_d)$. 
	\end{assumption}
	Due to the page limit, all proofs in this paper are given in the appendix.
	\section{Multi-point Gradient Estimator}\label{sec:multi_pt}
	
	
	In this section, we provide backgrounds on gradient estimators and propose the zeroth-order multi-point coordinate gradient estimator.
	
	\subsection{Backgrounds on Gradient Estimators}
	\label{subsec_background}
	Let $a$ denote the \emph{sampling radius} (also called the \emph{smoothing parameter}) of a random gradient estimator, and $\vu \sim p$ denote some random direction sampled from a distribution $p$. Denote $f_a$ as the \emph{smoothed} version of $f$ defined as
	$f_a(\vx) := \EE_{\vu \sim p'} [f(\vx+a \vu)],$
	where $p'$ is a certain distribution determined by $p$.
	All random gradients in this subsection are biased with respect to $\nabla f$ but unbiased with respect to $\nabla f_a$, satisfying
	$\EE_{\vu \sim p} [\tilde{\nabla} f(\vx)] = \nabla f_a (\vx)$,
	
	The \emph{$1$-point random gradient estimator} of $f$ has the form 
	\begin{equation}\label{eq:1pt_grad}
	\textstyle  \tilde{\nabla} f(\vx) := \frac{\phi(d)}{a} f(\vx+a \vu) \vu,
	\end{equation}
	where $\phi(d)$ is a dimension-dependent factor given by the distribution of $\vu$. If $p = \cN(\vzero, \vI)$, then $\phi(d) = 1$; if $p = \cU(\cS(\vzero, \vI))$ is the uniform distribution on the unit sphere, then $\phi(d) = d$. In practice, the $1$-point estimator in \eqref{eq:1pt_grad} is not commonly used due to high variance \cite{flaxman2004online}. 
	This motivates the \emph{$2$-point random gradient estimator} \cite{nesterov2017random, duchi2015optimal}
	\begin{equation}\label{eq:2pt_grad}
	\textstyle \tilde{\nabla} f(\vx) := \frac{\phi(d)}{a} (f(\vx+a \vu)-f(\vx)) \vu,
	\end{equation}
	where $\EE_{\vu \sim p}[\vu] = \vzero$ is required for unbiasedness to hold. The $2$-point estimator has the following upper bound of expected estimation error 
	\cite{berahas2021theoretical, liu2018zeroth}
	\begin{equation}\label{eq:2pt_grad_err}
	\textstyle \EE [\Vert \tilde{\nabla} f(\vx)-\nabla f(\vx) \Vert] = O(\sqrt{d}) \Vert \nabla f(\vx) \Vert + O \left( \frac{ad^{1.5}}{\phi(d)} \right).
	\end{equation}
	Note that the $O(\sqrt{d}) \Vert \nabla f(\vx) \Vert$ term in \eqref{eq:2pt_grad_err} does not vanish even if $a \rightarrow 0$. Mini-batch sampling 
	can be used to reduce the estimation error, leading to the \emph{multi-point random gradient estimator} \cite{duchi2015optimal,liu2018zeroth}
	\begin{equation}\label{eq:multi_grad}
	\tilde{\nabla} f(\vx) := \frac{\phi(d)}{a} \sum_{i=1}^{b} (f(\vx+a \vu_i)-f(\vx)) \vu_i,
	\end{equation}
	where $b$ is the mini-batch size, and $\{\vu_i\}_{i=1}^{b}$ are random directions drawn from some zero-mean distribution $p$. The multi-point estimator has the improved error bound \cite{berahas2021theoretical}
	\begin{align*}
	\textstyle 
	&\EE [\Vert \tilde{\nabla} f(\vx)-\nabla f(\vx) \Vert] \nonumber \\
	= & \ O \left( \sqrt{\frac{d}{b}} \right) \Vert \nabla f(\vx) \Vert + O \left( \frac{ad^{1.5}}{\phi(d)b} \right) + O \left( \frac{ad^{0.5}}{\phi(d)} \right).
	\end{align*} 
	To further reduce the estimation error, one can use \emph{coordinate-wise gradient} which requires $O(d)$ queries per gradient estimate. Existing works use \emph{forward difference} $\tilde{\nabla} f(\vx) := \frac{1}{a} \sum_{i=1}^{d} (f(\vx+a\ve_i)-f(\vx))\ve_i$ or \emph{central difference} $\tilde{\nabla} f(\vx) := \frac{1}{2a} \sum_{i=1}^{d} (f(\vx+a\ve_i)-f(\vx-a\ve_i))\ve_i$ as the coordinate gradient, where $\ve_i$ is the $i$th basis vector. Under the standard smoothness assumption, both forward difference and central difference have the error bounds \cite{kiefer1952stochastic, berahas2021theoretical, lian2016comprehensive}
	\begin{align*}
	&\EE [| \tilde{\nabla}_i f(\vx)-\nabla_i f(\vx) |] = O ( a ),\\ 
	&\EE [\Vert \tilde{\nabla} f(\vx)-\nabla f(\vx) \Vert] = O \left( a\sqrt{d} \right).
	\end{align*}
	
	\subsection{Zeroth-order Multi-point Coordinate Gradient Estimator}
	In this subsection, assuming $f$ to be coordinate $p$-smooth, we construct the zeroth-order multi-point coordinate gradient estimator (ZO-MCGE) $\tilde\nabla_i f(\vx)$ with $p = \max\{2(j-1), 2\}$ function value queries at $\vx+\frac{p}{2}a\ve_i, \dots, \vx+a\ve_i, \vx-a\ve_i, \dots, \vx-\frac{p}{2}a\ve_i$, and analyze its error bound. The main difference between our proposed ZO-MCGE and the estimators in Section \ref{subsec_background} is that the use of multi-point function evaluation allows for a better control for the gradient estimation error. We  observe numerically that using more points in the gradient estimator enables us to reach a higher accuracy; 
	see the logistic regression experiment in the Appendix.
	
	The following lemma directly follows from the coordinate $j$-smoothness of $f$. 
	
	\begin{lemma}\label{lemma:smooth_ineq} 
		Assume $f$ is $M_j$ coordinate $j$-smooth. Let $\nabla^l_i f(\vx) := \frac{\partial^l f(\vx)}{(\partial x_i)^l}$ be the $l$-th order derivative of $f$ at $\vx$ with respect to $x_i$. Then 
		\begin{align}\label{eq:smooth_ineq}
		&\left| f(\vx+b\ve_i) - f(x) - b \nabla_i f(\vx) - \dots - \frac{b^j}{j!}\nabla_i^j f(\vx) \right| \nonumber \\ 
		\le & \ \frac{M_j}{(j+1)!} |b|^{j+1}, \forall\, \vx \in \RR^d,\text{ and } b \in \RR.
		\end{align}
	\end{lemma}
	
	Let $a$ be the sampling radius. The following theorem states how to estimate the coordinate gradient $\nabla_i f(\cdot)$ of a $M_j$ coordinate $j$-smooth function $f$ by $p = \max\{2(j-1), 2\}$ queries at $\vx+\frac{p}{2}a\ve_i, \dots, \vx+a\ve_i, \vx-a\ve_i, \dots, \vx-\frac{p}{2}a\ve_i$, and provides the error bound.
	
	\begin{theorem}[multi-point coordinate gradient estimator]\label{thm:coor_grad_coef}
		Assume $f$ is $M_j$ coordinate $j$-smooth for some $j \in \ZZ^+$. Let $p = \max\{2(j-1),2\}$ and $m = \frac{p}{2}$. Define the $p$-point coordinate gradient estimator of $f$ with respect to some $i \in [d]$ as 
		\begin{align}\label{eq:coor_grad_coef}
		\textstyle \tilde{\nabla}_i f(\vx) := \ & C_{\frac{p}{2}} f(\vx+\frac{p}{2}ae_i) + \dots + C_1 f(\vx+ae_i) \nonumber \\ 
		&\hspace{-0.8cm}- C_1 f(\vx-ae_i) - \dots - C_{\frac{p}{2}} f(\vx-\frac{p}{2}ae_i),
		\vspace{-5pt}
		\end{align}
		\vspace{-5pt}
		where  
		\[
		\begin{bmatrix}
		C_1\\
		C_2\\
		\vdots\\
		C_{\frac{p}{2}}\\
		\end{bmatrix}
		=
		\begin{bmatrix}
		1&2&\cdots&\frac{p}{2}\\
		1&2^3&\cdots&(\frac{p}{2})^3\\
		&&\vdots&\\
		1&2^{p-1}&\cdots&(\frac{p}{2})^{p-1}\\
		\end{bmatrix}
		^{-1}
		\begin{bmatrix}
		\frac{1}{2a}\\
		0\\
		\vdots\\
		0\\
		\end{bmatrix}.
		\] \\
		Then we have the following error bound 
		\begin{equation}\label{eq:grad_err_bd}
		| \tilde{\nabla}_i f(\vx) - \nabla_i f(\vx) | \le \sum_{q=1}^{m} |C_q| \frac{M_j q^{j+1}}{(j+1)!}a^{j+1}.
		\end{equation}

	\end{theorem}

	\begin{remark}
		Theorem \ref{thm:coor_grad_coef} implies that if $f$ is $M_1$ coordinate $1$-smooth (which holds if $f$ is $M_1$-smooth in the standard notion) or $M_2$ coordinate $2$-smooth, then the coordinate gradient estimator given in \eqref{eq:coor_grad_coef} corresponds to the central difference $\tilde{\nabla}_i f(\vx) = \frac{1}{2a} (f(\vx+ae_i)-f(\vx-ae_i))$, with error bounds of $\frac{M_1}{2}a$ and $\frac{M_2}{6}a^2$ respectively, because $C_1=\frac{1}{2a}$.
		
		In general, we establish that under the $j$-th order smoothness assumption for some $j \in \ZZ^+$, 
		the error of the $\max\{2, 2(j-1)\}$-point coordinate gradient estimator is upper bounded by $O(a^j)$, where $a$ is the sampling radius. 
	\end{remark}
	\section{A Novel AL-based ZOM}\label{sec:alg}
	In this section, we present a novel ZOM for solving \eqref{eq:ncp} under the ALM framework, with each ALM subproblem approximately solved by an inexact proximal point method (iPPM).  
	
	The pseudocode of our AL-based ZOM for \eqref{eq:ncp} is shown in Algorithm~\ref{alg:ialm} that uses the following notations 
	\begin{subequations}\label{P3-eqn-33}
		\begin{align}
		&B_0=\max_{\vx\in\dom(h)}\max \big\{|f_0(\vx)|,\left\Vert \nabla g(\vx) \right\Vert\big\}, \nonumber \\ 
		&B_c = \max_{\vx\in\dom(h)}\Vert J_c(\vx) \Vert, \\
		&B_i=\max_{\vx\in\dom(h)}\max \big\{|c_i(\vx)|,\left\Vert \nabla c_i(\vx) \right\Vert\big\}, \forall\, i \in [l], \\
		& \textstyle \bar{B}_c = \sqrt{\sum_{i=1}^l B_i^2},\quad \bar{L} = \sqrt{\sum_{i=1}^lL_i^2}, \nonumber \\ 
		& \rho_c = \sum_{i=1}^{l} B_i \rho_i, \quad L_c = \sum_{i=1}^{l}B_i L_i + B_i^2, \label{eq:rho_c}
		\end{align}
	\end{subequations}
	where $\{\rho_i\}$ and $\{L_i\}$ are given in Assumption~\ref{assump:smooth-wc}. 
	
	\begin{algorithm}[h] 
		\caption{Zeroth-order inexact augmented Lagrangian method (ZO-iALM) for \eqref{eq:ncp}}\label{alg:ialm}
		\DontPrintSemicolon
		\textbf{Initialization:} choose $\vx^0\in\dom(f_0), \vy^0 = \vzero$, $\beta_0>0$ and $\sigma>1$\;
		\For{$k=0,1,\ldots,$}{
			Let $\beta_k=\beta_0\sigma^k$, $\phi(\cdot)=\cL_{\beta_k}(\cdot,y^k)-h(\cdot)$, and
			\begin{equation}\label{eq:hat_rho_k}
			\begin{aligned}
			&\hat{\rho}_k = \rho_0 + \bar L\|\vy^k\| + \beta_k \rho_c, \\ 
			&\hat{L}_k = L_0 + \bar L\|\vy^k\| + \beta_k L_c.
			\end{aligned} \vspace{-0.3cm} 
			\end{equation}\;
			$\vx^{k+1} \leftarrow \text{ZO-iPPM}(\phi,h,\vx^k,\hat{\rho}_k, \hat{L}_k, \vareps)$\;
			Update $\vy$ by
			\begin{align}
			\vy^{k+1} = ~\vy^k  + w_k \vc(\vx^{k+1}).\label{eq:alm-y}
			\end{align}
		}
		\setcounter{AlgoLine}{0}
		\SetKwProg{myproc}{subroutine}{}{}
		\myproc{\emph{ZO-iPPM{$(\phi,\psi,\vx^0,\rho, L_\phi,\vareps)$}}}{
			\For{$t=0,1,\ldots,$}{
				Let $G(\cdot)=\phi(\cdot)+\rho\Vert \cdot - \vx^t \Vert^2$\; 
				Obtain $\vx^{t+1}$ by a ZOM such that $\dist(\vzero, \partial(G+\psi)(\vx^{t+1}))\le \frac{\vareps}{4}$\;
				\textbf{if} $2\rho \Vert \vx^{t+1}-\vx^t \Vert \le \frac{\vareps}{2}$, \textbf{then} return $\vx^{t+1}$.
			}
		}
	\end{algorithm}
	
	Notice that Algorithm~\ref{alg:ialm} follows the standard framework of the ALM and uses ZO-iPPM to solve each ALM subproblem. In principle, one can use any ZOM as a subroutine to solve ALM subproblems, such as ZO-AdaMM \cite{chen2019zo} and ZO-proxSGD \cite{ghadimi2016mini}. However, the use of ZO-iPPM (together with our developed zeroth-order accelerated proximal coordinate update) not only leads to best known complexity results, but also gives better numerical performance, as we show in Section~\ref{sec:num}. 
	
	The proposed ZO-iALM is triple-looped. An algorithm with fewer loops would be preferable. However, we are not aware of any existing simpler ZOMs with the same theoretical guarantees as our method for solving functional constrained black-box optimization. An important future direction is to reduce the number of loops and achieve the same theoretical guarantees.
	Nevertheless, as we demonstrate in Section~\ref{sec:num}, our algorithm can be efficiently implemented without much difficulty. 
	Specifically, to have a good practical performance, all parameters except the smoothness constant do not require much tuning at all, and most can be constant across different problems. Even with triple loops, the proposed ZO-iALM performs well numerically. Furthermore, some existing FOMs are also triple-looped and can perform better than double-looped FOMs; see 
	\cite{li2021rate} for example. 
	
	The kernel problems that we solve within the iPPM are strongly-convex composite problems. Below, we design a zeroth-order accelerated proximal coordinate update (ZO-APCU) method.
	
	\subsection{Core subsolver: ZO-APCU}\label{sec:apcg}
	In this subsection, we give our core ZO subsolver, called ZO-APCU, to obtain $\vx^{t+1}$ in the ZO-iPPM subroutine. Though ZO-APCU will be used for solving subproblems of our proposed ZOM for \eqref{eq:ncp}, it has its own merit and appears to be the first proximal coordinate update method \emph{with acceleration} by only using function values of the smooth part. It solves strongly-convex composite problems in the form of
	\begin{equation}\label{eq:comp-prob}
	\min_{\vx\in\RR^d}~ F(\vx):=G(\vx)+H(\vx),
	\end{equation}
	where $G$ is a \emph{black-box} $\mu$-strongly convex and $L$-smooth function, and $H$ is a \emph{white-box} closed convex function. 
	
	We make the following assumptions on $G$ and $H$.
	\begin{assumption}[coordinate smooth]
		\label{assump:coor_smooth}
		$G$ is $M_j$ coordinate $j$-smooth, for some $j \in \ZZ^+$.
	\end{assumption}
	
	Note that if $G$ is $L$-smooth, Assumption \ref{assump:coor_smooth} must hold for $j = 1$ and $M_1 = L$.
	
	\begin{assumption}\label{assump:sep_apcg}
		The function $H$ is coordinate-separable, i.e., $H(\vx) = \sum_{i=1}^{d} H_i(x_i)$, where each $H_i(\cdot)$ is convex.
	\end{assumption}
	
	The pseudocode of ZO-APCU is shown in Algorithm~\ref{alg:apcg}, with its equivalent and efficient implementation (which avoids full-dimensional vector operations) given in the Appendix. The design is inspired from the APCG method in \cite{lin2014accelerated}. A zeroth-order accelerated random search (ZO-ARS) method has been designed in \cite{nesterov2017random} to solve \eqref{eq:comp-prob}. Although our ZO-APCU has the same-order query complexity as ZO-ARS, it significantly outperforms ZO-ARS in practice (see Section \ref{sec:num}), because ZO-APCU exploits the coordinate-structure and uses more accurate coordinate gradient estimator. 
	
	

	
	\begin{algorithm}[h] 
		\caption{Zeroth-order accelerated proximal coordinate update for \eqref{eq:comp-prob}: $\text{ZO-APCU}(G,H,\mu, L,\vareps)$}\label{alg:apcg}
		\DontPrintSemicolon
		\textbf{Input:} $\vx^{0} \in \dom(H)$, tolerance $\vareps$, smoothness $L$, strong convexity $\mu$, and epoch length $l$. \;
		\textbf{Initialization:} $\vz^0 = \vx^0, \alpha = \frac{1}{d}\sqrt{\frac{\mu}{L}}$ \;
		\For{$k=0,1,\ldots $}{
			Let $y^k = \frac{\vx^k+\alpha \vz^k}{1+\alpha}$\; 
			Sample $i_k \in [d]$ uniformly; compute
			$\tilde{\nabla}_{i_k} G(\vy^k)$ such that $\Vert \tilde{\nabla}_{i_k} G(\vy^k) - \nabla_{i_k} G(\vy^k) \Vert \le E_{i_k}.$\;
			Compute 
			$\vz^{k+1} =  \argmin_{\vx \in \RR^d} \{ \frac{dL\alpha}{2} \Vert \vx-(1-\alpha)\vz^k-\alpha \vy^k \Vert^2 + \langle \tilde{\nabla}_{i_k}G(\vy^k), \vx_{i_k}-\vy_{i_k}^k \rangle+H_{i_k}(\vx_{i_k}) \}$.\; 
			$\vx^{k+1} = \vy^k + d\alpha(\vz^{k+1}-\vz^k)+d\alpha^2 (\vz^k-\vy^k)$.\; 
			\If{$k+1 \equiv 0 \pmod l$}{ 
				Compute $\tilde{\nabla} G(\vx^{k+1})$ such that $\Vert \tilde{\nabla} G(\vx^{k+1}) - \nabla G(\vx^{k+1}) \Vert \le E$\; 
				$\hat{\vx}^{k+1} = \argmin_{\vx \in \RR^d} \{ \langle \tilde{\nabla} G(\vx^{k+1}), \vx-\vx^{k+1} \rangle + \frac{L}{2} \Vert \vx-\vx^{k+1} \Vert^2 + H(\vx) \}$\;
				\textbf{Return} $\hat{\vx}^{k+1}$ and \textbf{stop} if $\dist \big(\vzero, \tilde{\nabla} G(\hat{\vx}^{k+1}) + \partial H(\hat{\vx}^{k+1})\big) \le \frac{3\vareps}{4}$. 
			}
		}
	\end{algorithm}

	In Algorithm \ref{alg:apcg}, to obtain the required (coordinate) gradient estimates, we use the $p$-point coordinate gradient estimator defined in \eqref{eq:coor_grad_coef}, where $p = \max\{2(j-1),2\}$. Let 
	\begin{equation}\label{eq:E}
	E_i = \sum_{q=1}^{m}|C_q|\frac{M_j q^{j+1}}{(j+1)!}a^{j+1}, \forall i \in [d];\  
	E = \sqrt{\sum_{i=1}^{d}E_i^2},
	\end{equation}
	where $m = \frac{p}{2}$ and $a$ is the sampling radius. 
	By Theorem~\ref{thm:coor_grad_coef}, $E$ and $E_i$ are upper bounds of the gradient estimation errors for $\nabla G(\cdot)$ and $\nabla_i G(\cdot)$. Let $\bar{\vareps} = \frac{\mu}{512L} \vareps^2$. 
	We choose $a>0$ and $p$ such that the error bounds $E$ and $\{E_i\}_{i=1}^{d}$ in \eqref{eq:E} satisfy 
	\begin{equation}\label{eq:small_grad_err}
	2L\sqrt{\frac{2ED}{\mu}} + E \le \frac{\vareps}{4},\quad
	ED+\sum_{i=1}^{d}E_i D_i \le \frac{\bar{\vareps}}{2}.
	\end{equation}
	
	
	\subsection{Complexity Results}\label{sec:complexity}
	
	In this subsection, we establish the total query complexity result of Algorithm \ref{alg:ialm}. We first show that the core subsolver ZO-APCU can produce $\vx^{t+1}$ desired in the ZO-iPPM subroutine.
	The theorem below gives the complexity result of ZO-APCU to produce an expected $\vareps$-stationary point of \eqref{eq:comp-prob}.
	The proof is highly nontrivial and given in the appendix. 
	
	\begin{theorem}\label{thm:apcg_cplx_expected}
		Let $\{\vx^k\}, \{\hat{\vx}^k\}$ be generated from Algorithm \ref{alg:apcg}. 
		Suppose the gradient error bounds $E$ and $\{E_i\}_{i=1}^{d}$ satisfy \eqref{eq:small_grad_err}. Let $\bar{\vareps} = \frac{\mu}{512L^2}\vareps^2$.
		Then $T = \left\lceil d\sqrt{\frac{L}{\mu}} \log\frac{2(F(\vx^0)-F^*)+\mu \Vert \vx^0-\vx^* \Vert^2}{\bar{\vareps}} \right\rceil$ iterations of ZO-APCG suffice to generate $\hat{\vx}^T$ satisfying 
		$\EE [\dist(\vzero, \partial F(\hat{\vx}^T))] \le \vareps$.
	\end{theorem}
	
	Theorem~\ref{thm:apcg_cplx_expected} only guarantees that the output $\hat{\vx}^T$ nearly satisfies the stationarity condition \emph{in expectation}. In order to show the complexity results of Algorithm~\ref{alg:ialm}, we need, in Line 4 of ZO-iALM, the iterate $\vx^{k+1}$ obtained from ZO-iPPM \emph{deterministically} satisfies the near-stationarity condition of $\cL_{\beta_k}(\cdot,y^k)$ so that we can uniformly bound the AL function at the generated iterates. For this technical reason, we will require the output from Algorithm~\ref{alg:apcg} to satisfy the near-stationarity condition \emph{deterministically} instead of in an expectation sense.
	Theorem~\ref{thm:exp-cplx} below serves as a bridge to convert deterministic iteration bound until expected convergence to expected iteration bound until deterministic convergence, by only sacrificing a $\log$ factor in the iteration bound. 
	The result is not difficult to prove but is essential in our complexity analysis of ZO-iALM. 
	
	\begin{theorem}[expected complexity]\label{thm:exp-cplx}
		For a sequence of nonnegative random numbers $\{q_k\}_{k=1}^{\infty}$, suppose $\EE [q_k]\leq C \eta^k, \forall k \geq 1$ for some $\eta \in (0,1)$ and $C > 0$.
		Given $\vareps > 0$, define $K(\vareps) = \min_{k \in \ZZ^+} \{ k:  q_k \leq \vareps\}$.
		Then $\EE [K(\vareps)] \leq \frac{2-\eta}{1-\eta} \log\frac{C}{\vareps (1-\eta)^2}+3-\eta$.
	\end{theorem}

	By Theorem \ref{thm:apcg_cplx_expected} (and its proof for linear convergence) and Theorem \ref{thm:exp-cplx}, we have the following 
	expected iteration complexity result of Algorithm \ref{alg:apcg} until deterministic convergence.
	\begin{corollary}\label{cor:apcg_cplx}
		Under the same assumptions as Theorem \ref{thm:apcg_cplx_expected}, $T$ iterations of ZO-APCU are enough to generate $\hat{\vx}^T$ satisfying 
		$\dist(\vzero, \partial F(\hat{\vx}^T)) \le \vareps$, where $\EE[T] = \tilde{O} \left( d\sqrt{\frac{L}{\mu}} \right)$.
	\end{corollary}

	Relying on Corollary \ref{cor:apcg_cplx}, the next theorem gives the complexity result of the subroutine ZO-iPPM applied on the nonconvex composite problem 
	\begin{equation}\label{eq:nc_prob}
	\Phi^*= \min_{\vx\in\RR^d} ~\big\{\Phi(\vx) := \phi(\vx)+\psi(\vx)\big\},
	\end{equation}
	where $\phi$ is a \emph{black-box} $L_\phi$-smooth and $\rho$-weakly convex function, and $\psi$ is a \emph{white-box} closed convex function.
	\begin{theorem}\label{thm:ippm-compl}
		Suppose $\Phi^*$ in \eqref{eq:nc_prob} is finite. Then the subroutine ZO-iPPM in Algorithm~\ref{alg:ialm} must stop within $T$ iterations, where
		$\textstyle T = \left\lceil \frac{32 \rho}{\vareps^2} (\Phi(\vx^0)-\Phi^*) \right\rceil.$
		The output $\vx^{T}$ must 
		satisfy $\dist(\vzero, \partial \Phi(\vx^T)) \le \vareps$. In addition, if $\dom(\psi)$ has diameter $D_\psi < \infty$ and ZO-APCU is applied to find each $\vx^{t+1}$ in ZO-iPPM, then the expected total query complexity is $\tilde{O}\left( \frac{d\sqrt{\rho L_\phi}}{\vareps^2} [\Phi(\vx^0)-\Phi^*] \log \frac{D_\psi}{\vareps} \right)$.
	\end{theorem} 

	Now we are ready to establish the query complexity of the proposed ZO-iALM. Due to the difficulty of the possibly nonconvex constraints, a certain regularity condition must be made in order to guarantee (near) feasibility in a polynomial time. Following \cite{li2021rate, lin2019inexact-pp, sahin2019inexact} that study FOMs, we assume the following regularity condition on \eqref{eq:ncp}. 
	
	\begin{assumption}[regularity] \label{assumption:regularity}
		There is some $v > 0$ such that for any $k\ge 1$,
		\begin{equation}\label{eq:regularity}
		\textstyle v\Vert \vc(\vx^k) \Vert \le \dist \left(-J_c(\vx^k)^{\top} \vc(\vx^k), \frac{\partial h(\vx^k)}{\beta_{k-1}} \right).
		\end{equation}
	\end{assumption}
	
	\begin{remark}
		Notice that we only need the existence of $v$ in Assumption \ref{assumption:regularity} but do not need to know its value in our algorithm. The assumption ensures that a near-stationary point of the AL function is near feasible. In \cite{li2021rate}, the regularity condition is proven to hold for all affine-equality constrained problems possibly with either an additional polyhedral or ball constraint set. Moreover, several nonconvex examples satisfying Assumption \ref{assumption:regularity} are given in \cite{lin2019inexact-pp, sahin2019inexact}.
		
		With Assumption \ref{assumption:regularity}, we can simply solve a quadratic-penalty problem of \eqref{eq:ncp} with a large enough penalty parameter, in order to find a near-KKT point of \eqref{eq:ncp}. 
		However, this approach is numerically much slower than the iALM framework in Algorithm \ref{alg:ialm}; see the tests in \cite{li2021rate} for example. 
	\end{remark}
	
	\begin{remark}
		To solve the nonconvex constrained problem \eqref{eq:ncp}, a few existing works about FOMs have made key assumptions different from Assumption \ref{assumption:regularity}. For example, the uniform Slater's condition was assumed in \cite{ma2020quadratically}, and a strong MFCQ condition was assumed in \cite{boob2019stochastic}. These assumptions are neither strictly stronger nor strictly weaker than Assumption \ref{assumption:regularity}.
	\end{remark}
	
	The theorem below gives the total query complexity of ZO-iALM with general dual step sizes. 
	
	\begin{theorem}[total complexity of ZO-iALM]\label{thm:ialm-cplx-2}
		Suppose that Assumptions~\ref{assump:smooth-wc}, \ref{assump:composite}, and \ref{assumption:regularity} hold. In Algorithm~\ref{alg:ialm}, for some fixed $q \in \ZZ^+ \cup \{0\}$ and $M > 0$, let 
		$w_k = \frac{M(k+1)^q}{\Vert \vc(\vx^{k+1}) \Vert}, \forall k \ge 0.$
		Then given $\vareps>0$, Algorithm~\ref{alg:ialm} can produce an $\vareps$-KKT solution of \eqref{eq:ncp} 
		with $\tilde{O}(d\vareps^{-3})$ queries to $g$ and $\vc$ in expectation, by using Algorithm~\ref{alg:apcg} to find each $\vx^{t+1}$ in ZO-iPPM. In addition, if $\vc(\vx)=\vA\vx-\vb$, then $\tilde{O}(d\vareps^{-\frac{5}{2}})$ queries in expectation are enough for Algorithm~\ref{alg:ialm} to produce an $\vareps$-KKT solution of \eqref{eq:ncp}. 
	\end{theorem}
	
	\begin{remark}
		The results in Theorem~\ref{thm:ialm-cplx-2} are novel. To the best of our knowledge, they are the first such results for ZOMs on solving functional constrained black-box optimization. The order-dependence on $\vareps$ matches with the best-known results for FOMs on solving nonconvex composite optimization with convex or nonconvex constraints, e.g., see \cite{lin2019inexact-pp, melo2020iteration, li2021rate}. For the affine-constrained case, we conjecture that the $\tilde{O}(d\vareps^{-\frac{5}{2}})$ query complexity may be reduced to $\tilde{O}(d\vareps^{-2})$ if the nonsmooth term $h$ has some special structure.
		For example, the FOM in \cite{zhang2020global} achieves an $\tilde{O}(\vareps^{-2})$ complexity result for solving affine-constrained smooth nonconvex optimization. It would be an interesting future work to combine the technique in \cite{zhang2020global} with ours to develop a ZOM that can achieve $\tilde{O}(d\vareps^{-2})$ query complexity. 
	\end{remark} 

	\section{Numerical Results}\label{sec:num}
	
	In this section, we conduct numerical experiments to demonstrate the performance of our proposed ZO-iALM. We consider the problem of resource allocation in sensor networks and the adversarial example generation problem. All the tests were performed in MATLAB 2019b on a Macbook Pro with 4 cores and 16GB memory. Due to the page limitation, we put additional numerical experiments in the Appendix. They are on nonconvex linearly-constrained quadratic programs (LCQP), on the unconstrained strongly-convex quadratic programs (USCQP) to test the core solver ZO-APCU, and on the logistic regression to test different multi-point coordinate gradient estimators. 
	We emphasize here that the proposed ZO-APCU requires significantly fewer queries to reach a near-stationary point to the USCQP problem compared to a few existing methods, and that the use of more points in coordinate gradient estimator can lead to higher accuracy.
	
	
	\subsection{Resource Allocation in Sensor Networks}\label{subsec:exp2}
	In this subsection, we test our proposed ZO-iALM on the resource allocation problem in sensor networks
	\cite{liu2016sensor}. The problem aims at minimizing the estimation error of a random vector with a Gaussian prior probability density function, subject to a constraint on the total number of sensor activations. It can be formulated as 
	\begin{equation}\label{eq:sensor_sel_orig}
	\begin{aligned}
	\min_{\vw \in \RR^d} & \ \tr (\Sigma^{-1} + \vH^{\top}(\vw \vw^{\top} \circ \vR^{-1})\vH)^{-1}, \\ 
	\st & \ \vone^{\top}\vw \le s, \vw \in \{0,1\}^d,
	\end{aligned}
	\end{equation}
	where each $w_i \in \{0,1\}$ denotes whether the $i$th sensor is selected, $\vH \in \RR^{d \times d}$ is the observation matrix, $\Sigma \in \RR^{d \times d}$ is the MSE source statistics, and $\vR \in \RR^{d \times d}$ is the noise covariance matrix. We assume that $\Sigma$ and $\vR$ are symmetric, and $\vR$ has small off-diagonal entries. Details of the formulation \eqref{eq:sensor_sel_orig} can be found in \cite{liu2016sensor}. 
	
	ZO optimization methods have been applied in the literature to problem \eqref{eq:sensor_sel_orig}, in order to avoid the involved first-order gradient computation \cite{liu2018zeroth}. The use of a ZO solver enables the design of resource management with least prior knowledge, e.g., without having access to the 
	sensing model information encoded in $\mathbf H$.
	The constraint $\vw \in \{0,1\}^d$ is combinatorial. 
	Below, we rewrite the 0-1 constraint to $\vw^2-\vw = \vzero$ and also incorporate the constraint $\vone^{\top}\vw \le s$ into the objective by introducing a (fixed) multiplier $\lambda>0$. More precisely, we apply our ZO-iALM to the problem: 
	\begin{equation}\label{eq:sensor_sel}
	\begin{aligned}
	\min_{\vw \in \RR^d} & \ \tr (\Sigma^{-1} + \vH^{\top}(\vw \vw^{\top} \circ \vR^{-1})\vH)^{-1} + \lambda \vone^{\top}\vw, \\ 
	\st & \ \vw^2-\vw = \vzero.
	\end{aligned}
	\end{equation}
	
	Since no existing ZOMs are able to handle nonconvex constrained problems, we compare the proposed ZO-iALM to two other methods that replace our ZO-iPPM subroutine with ZO-AdaMM \cite{chen2019zo} and ZO-ProxSGD \cite{ghadimi2016mini} respectively. 
	
	We set $d = 80$, $\lambda = 0.5$, and $\vareps = 0.5$. Following \cite{liu2016sensor}, we construct $\vH = \frac{1}{2}(\bar{\vH}+\bar{\vH}^\top)$ with each entry of $\bar{\vH} \in \RR^{d \times d}$ generated from the uniform distribution $\cU(0,1)$, $\vR = (\frac{1}{2}(\bar{\vR}+\bar{\vR}^\top))^{-1}$ with each entry of $\bar{\vR} \in \RR^{d \times d}$ generated from $\cU(0,10^{-3})$, and $\Sigma = \vI$. In each call to the ZO-iPPM subroutine, we set the smoothness parameter to $\hat L_k = 50 + 0.3\beta_k$. We tune the parameters of ZO-AdaMM to $\alpha = 1, \beta_1 = 0.75, \beta_2 = 1$, and fix the step size to $0.01$ in ZO-ProxSGD. For each method, we choose $a = 10^{-6}$ as the sampling radius and $w_k = \frac{1}{\Vert \vc(\vx^k) \Vert}$ as the dual step size.
	
	In Figure \ref{fig:plot}, we compare the primal residual trajectories of the proposed ZO-iALM, and the iALM with subroutine ZO-AdaMM in \cite{chen2019zo} and ZO-ProxSGD in \cite{ghadimi2016mini}. The dual residuals by all compared methods are below the error tolerance $\vareps$ at the end of each outer loop. In Table~\ref{table:sensor} in the appendix, we also report the primal residual, dual residual, running time (in seconds), and the query count, shortened as \verb|pres|, \verb|dres|, \verb|time|, and \verb|#Obj|, for each method. From the results, we conclude that the proposed ZO-iALM with any of the three subroutines is able to reach an $\vareps$-KKT point to the resource allocation problem \eqref{eq:sensor_sel}. Moreover, the proposed ZO-iPPM subroutine requires fewer queries than other compared ZOMs to find a specified-accurate stationary point to the nonconvex subproblems.
	
	\begin{figure}[t] 
		\begin{center}
			Resource Allocation in Sensor Networks \\
			\includegraphics[width=0.3\textwidth]{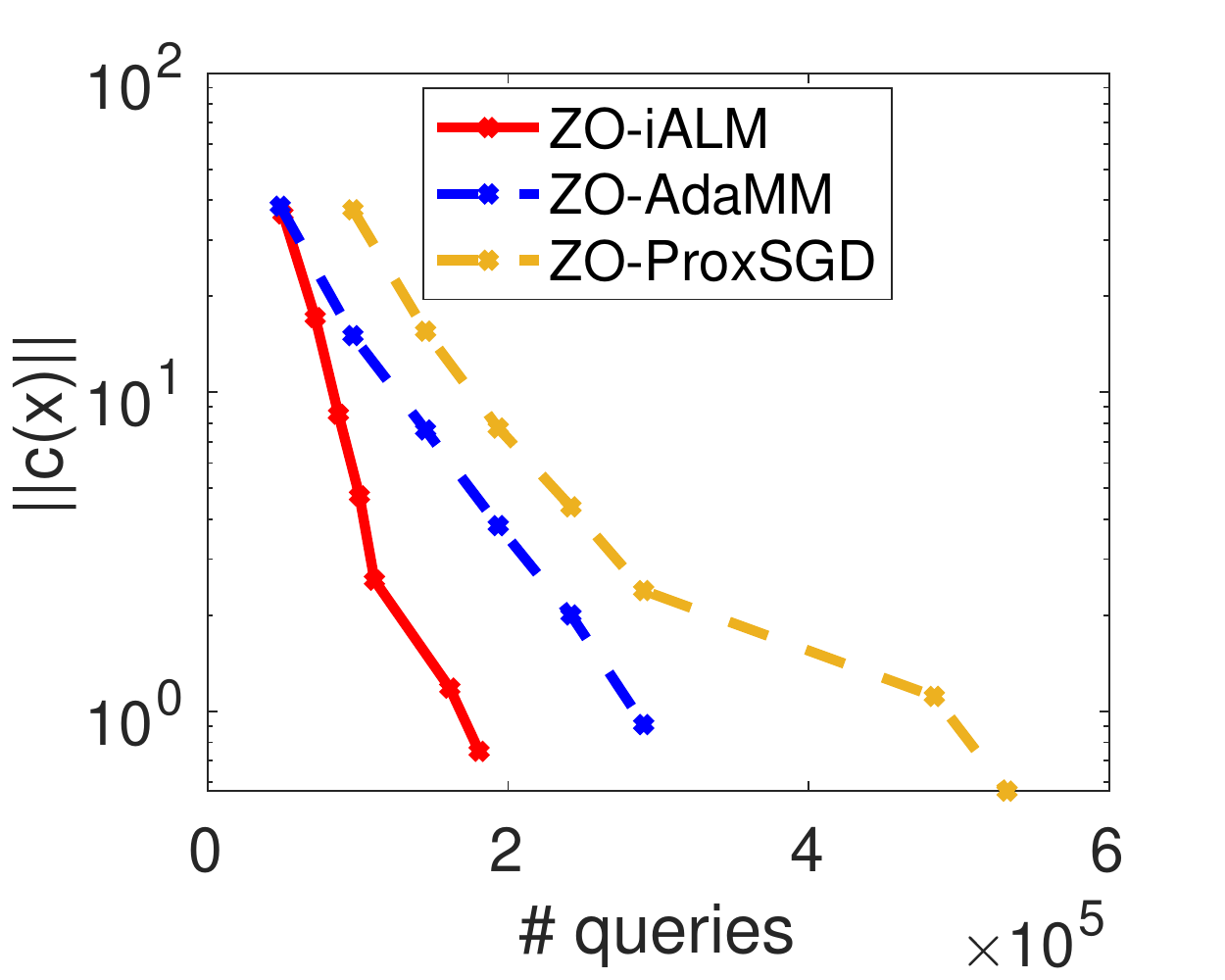} 
		\end{center}
		\vspace{-4mm}
		\caption{Comparison of iALM on solving \eqref{eq:sensor_sel} with different subroutines: the proposed ZO-iALM, ZO-AdaMM in \cite{chen2019zo}, and ZO-ProxSGD in \cite{ghadimi2016mini}. The plots show primal residuals. The markers denote the outer iterations in iALM.
			Dual residuals for all methods are below the given tolerance $\vareps$.}\label{fig:plot}
	\end{figure}
	
	\subsection{Adversarial Example Generation}\label{subsec:atk}
	The problem of adversarial example generation for a black-box regression model \cite{liu2020min} under both $L_0$ and $L_{\infty}$-norm constraints can be formulated as
	\begin{equation}\label{eq:adv_ex_orig}
	\max_{\Vert \Delta \Vert_{\infty} \le \vareps_{\infty}} f_{\theta}(\vx + \Delta), \st \Vert \Delta \Vert_0 \le \vareps_0,
	\end{equation}
	where $f_\theta(\cdot)$ is a loss function of a black-box regression model parameterized by $\theta$ that is trained over the dataset $\vx = [\vx_1^{\top};\dots;\vx_m^{\top}] \in \RR^{m \times d}$, $\Delta \in \RR^d$ is the data perturbation, and $\vx+\Delta$ denotes adding $\Delta$ to each $\vx_i$.
	
	The constraint $\Vert \Delta \Vert_0 \le \vareps_0$ is combinatorial. To relax it to a continuous one, we introduce a binary vector $\hat{M}$ as a mask and put the constraint onto $\hat{M}$. More precisely, replace $\Delta$ in \eqref{eq:adv_ex_orig} by $\hat{M} \circ \Delta$, where $\circ$ denotes the Hadamard (component-wise) product. Then the constraint $\Vert \Delta \Vert_0 \le \vareps_0$ is relaxed to $\hat{M}_i\in \{0,1\},\forall\, i$ and $\vone^{\top} \hat{M}\le \vareps_0$. By further incorporating the constraint $\vone^{\top} \hat{M}\le \vareps_0$ into the objective by introducing a (fixed) multiplier $-\lambda <0$ and rewrite $\hat{M}_i\in \{0,1\},\forall\, i$ into $\hat{M}^2 - \hat{M} = \vzero$, where $\hat{M}^2$ denotes the component-wise square of $\hat{M}$, we have the following reformulation:
	\begin{equation}\label{eq:adv_ex}
	\max_{\substack{\hat{M}, \Delta \in \RR^d \\ \Vert \Delta \Vert_{\infty} \le \vareps_{\infty} }} f_{\theta}(\vx_0 + \hat{M} \circ \Delta) - \lambda \vone^{\top} \hat{M}, \st \hat{M}^2 - \hat{M} = \vzero.
	\end{equation}
	
	We test the proposed ZO-iALM on the adversarial example generation problem \eqref{eq:adv_ex}.
	In the test, we use the ovarian cancer dataset \cite{conrads2004high, petricoin2002use} that are from $m = 216$ patients. Each data point has $d = 4,000$ features and a label indicating whether the corresponding patient has ovarian cancer. We first use MATLAB's built-in lasso function (with $\lambda = 0.01$) to train a LASSO regression model parameterized by $\theta$. With the trained model, we treat the regression loss $f_{\theta}(\cdot)$ as a ZO oracle and perform black-box attack on it. Let $\vx \in \RR^{m \times d}$ denote the data matrix. We then solve the ZO formulation \eqref{eq:adv_ex} to find an adversarial perturbation $M \circ \Delta$ to each row of $\vx$ that near-maximally increases the regression loss $f_{\theta}(\cdot)$. 
	In \eqref{eq:adv_ex}, we set $\lambda = 0.01$ and $\vareps_{\infty} = 0.1$. Due to the large variable dimension, we set $\vareps = 1$ in stopping conditions. 
	
	The same as the previous test, we compare the proposed ZO-iALM to two other methods that replace our ZO-iPPM subroutine with ZO-AdaMM \cite{chen2019zo} and ZO-ProxSGD \cite{ghadimi2016mini} respectively. In each method, we set $a = 10^{-6}$ as the sampling radius and $w_k = \frac{1}{\Vert \vc(\vx^k) \Vert}$ as the dual step size.
	
	Let $(\hat{M},\Delta)$ be one iterate obtained by one method on solving \eqref{eq:adv_ex}. Then $\tilde\Delta \gets \hat{M} \circ \Delta$ is the data perturbation. To recover the solution to \eqref{eq:adv_ex_orig}, we project $\tilde\Delta$ to the set $\{\Delta: \Vert \Delta \Vert_0 \le 20, \Vert \Delta \Vert_{\infty} \le 0.1\}$. In Figure \ref{fig:atk}, we plot the trajectory of the loss objective 
	$f_{\theta}$ by all methods at the processed iterates of perturbed data. 
	From the results, we see that the data perturbation created by the proposed ZO-iALM 
	increases the loss function faster (namely, creates more successful attacks) than other compared methods. 
	
	\begin{figure}[t] 
		\begin{center}
			Adversarial Example Generation \\
			\includegraphics[width=0.3\textwidth]{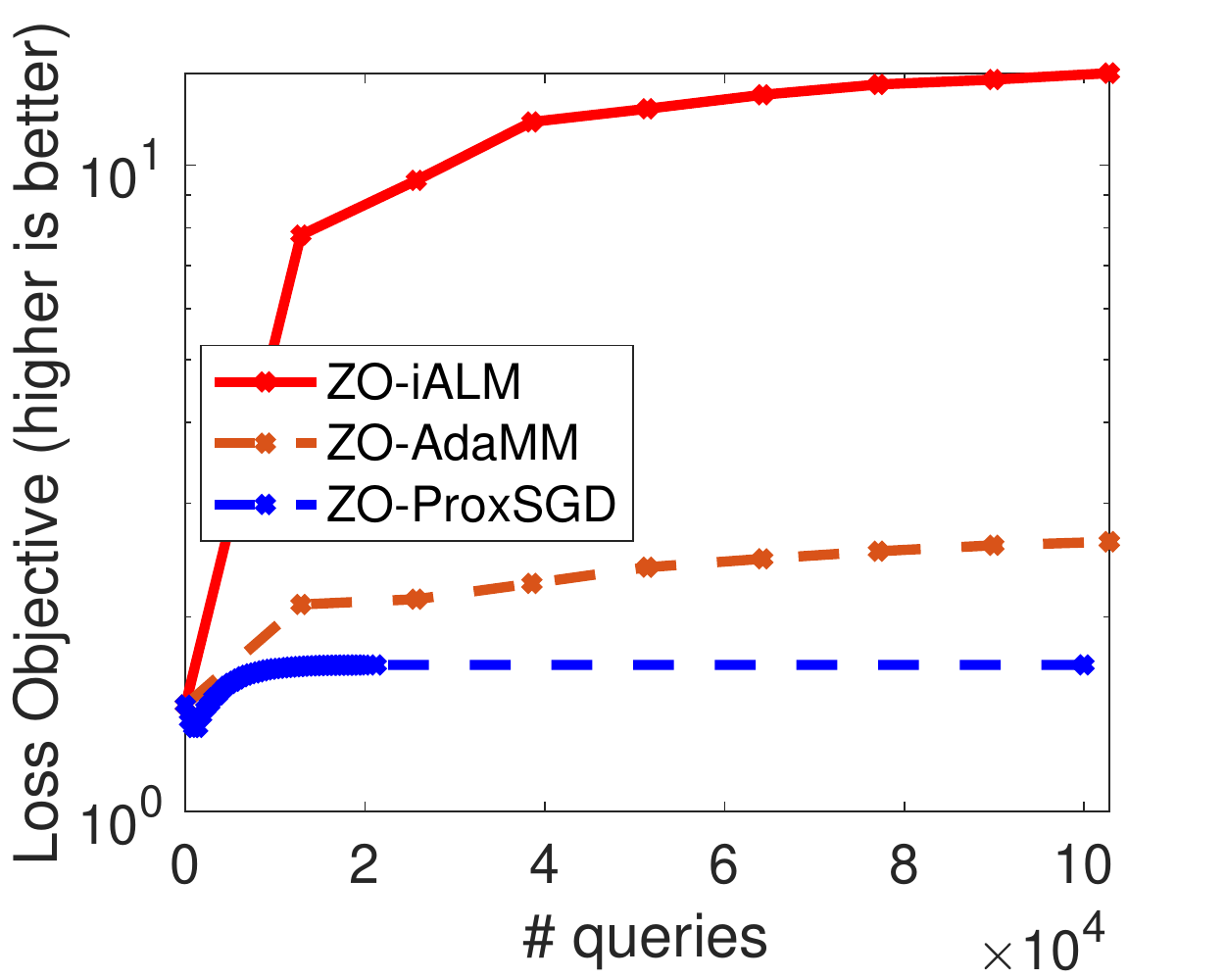} 
		\end{center}
		\vspace{-4mm}
		\caption{Comparison of iALM on solving \eqref{eq:adv_ex} with different subroutines: the proposed ZO-iALM, ZO-AdaMM in \cite{chen2019zo}, and ZO-ProxSGD in \cite{ghadimi2016mini}. The plots show the loss objective that we attack under the same $L_0$ and $L_{\infty}$ constraints.}\label{fig:atk}
	\end{figure}
	\color{black} 
	
	\section{Conclusion}
	
	In this paper, we propose a novel zeroth-order inexact augmented Lagrangian method (ZO-iALM) to solve black-box optimization problems that involve a composite (i.e., smooth+nonsmooth) objective and nonlinear functional constraints. The kernel subproblems that we solve during the ZO-iALM are black-box strongly-convex composite problems with coordinate structure. To most efficiently solve these subproblems, we design a zeroth-order accelerated proximal coordinate update (ZO-APCU) method. In addition, in order to be able to produce high-accurate solutions, we give a new multi-point coordinate gradient estimator and use it in our designed ZO-APCU. All our proposed zeroth-order methods achieve similar-order complexity results as the best-known results obtained by first-order methods, with a difference up to a factor of variable dimension. 
	Besides the novel and best theoretical results, our proposed ZO-iALM can also perform well numerically, which is demonstrated by experiments on practical machine learning tasks and classical optimization problems. 
	
	\clearpage 
	\newpage 
	
	\section*{Acknowledgements}
	
	
	This work was supported by the Rensselaer-IBM AI Research Collaboration (\url{http://airc.rpi.edu}), part of the IBM AI Horizons Network (\url{http://ibm.biz/AIHorizons}). The work of Y. Xu is partly supported by NSF Award 2053493 and the RPI-IBM AIRC faculty fund.
	
	\bibliography{optim}

\begin{thebibliography}{32}
\providecommand{\natexlab}[1]{#1}

\bibitem[{Berahas et~al.(2021)Berahas, Cao, Choromanski, and
  Scheinberg}]{berahas2021theoretical}
Berahas, A.~S.; Cao, L.; Choromanski, K.; and Scheinberg, K. 2021.
\newblock A theoretical and empirical comparison of gradient approximations in
  derivative-free optimization.
\newblock \emph{Foundations of Computational Mathematics}, 1--54.

\bibitem[{Boob, Deng, and Lan(2019)}]{boob2019stochastic}
Boob, D.; Deng, Q.; and Lan, G. 2019.
\newblock Stochastic first-order methods for convex and nonconvex functional
  constrained optimization.
\newblock \emph{arXiv preprint arXiv:1908.02734}.

\bibitem[{Chen et~al.(2017)Chen, Zhang, Sharma, Yi, and Hsieh}]{chen2017zoo}
Chen, P.-Y.; Zhang, H.; Sharma, Y.; Yi, J.; and Hsieh, C.-J. 2017.
\newblock {ZOO}: Zeroth Order Optimization Based Black-box Attacks to Deep
  Neural Networks Without Training Substitute Models.
\newblock In \emph{ACM Workshop on Artificial Intelligence and Security},
  15--26.

\bibitem[{Chen et~al.(2019)Chen, Liu, Xu, Li, Lin, Hong, and Cox}]{chen2019zo}
Chen, X.; Liu, S.; Xu, K.; Li, X.; Lin, X.; Hong, M.; and Cox, D. 2019.
\newblock Zo-adamm: Zeroth-order adaptive momentum method for black-box
  optimization.
\newblock \emph{arXiv preprint arXiv:1910.06513}.

\bibitem[{Conrads et~al.(2004)Conrads, Fusaro, Ross, Johann, Rajapakse, Hitt,
  Steinberg, Kohn, Fishman, Whitely et~al.}]{conrads2004high}
Conrads, T.~P.; Fusaro, V.~A.; Ross, S.; Johann, D.; Rajapakse, V.; Hitt,
  B.~A.; Steinberg, S.~M.; Kohn, E.~C.; Fishman, D.~A.; Whitely, G.; et~al.
  2004.
\newblock High-resolution serum proteomic features for ovarian cancer
  detection.
\newblock \emph{Endocrine-related cancer}, 11(2): 163--178.

\bibitem[{Dua and Graff(2017)}]{Dua:2019}
Dua, D.; and Graff, C. 2017.
\newblock {UCI} Machine Learning Repository.

\bibitem[{Duchi et~al.(2015)Duchi, Jordan, Wainwright, and
  Wibisono}]{duchi2015optimal}
Duchi, J.~C.; Jordan, M.~I.; Wainwright, M.~J.; and Wibisono, A. 2015.
\newblock Optimal rates for zero-order convex optimization: The power of two
  function evaluations.
\newblock \emph{IEEE Transactions on Information Theory}, 61(5): 2788--2806.

\bibitem[{Flaxman, Kalai, and McMahan(2004)}]{flaxman2004online}
Flaxman, A.~D.; Kalai, A.~T.; and McMahan, H.~B. 2004.
\newblock Online convex optimization in the bandit setting: gradient descent
  without a gradient.
\newblock \emph{arXiv preprint cs/0408007}.

\bibitem[{Ghadimi, Lan, and Zhang(2016)}]{ghadimi2016mini}
Ghadimi, S.; Lan, G.; and Zhang, H. 2016.
\newblock Mini-batch stochastic approximation methods for nonconvex stochastic
  composite optimization.
\newblock \emph{Mathematical Programming}, 155(1-2): 267--305.

\bibitem[{Kiefer, Wolfowitz et~al.(1952)}]{kiefer1952stochastic}
Kiefer, J.; Wolfowitz, J.; et~al. 1952.
\newblock Stochastic estimation of the maximum of a regression function.
\newblock \emph{The Annals of Mathematical Statistics}, 23(3): 462--466.

\bibitem[{Li and Qu(2021)}]{li2021inexact}
Li, F.; and Qu, Z. 2021.
\newblock An inexact proximal augmented Lagrangian framework with arbitrary
  linearly convergent inner solver for composite convex optimization.
\newblock \emph{Mathematical Programming Computation}, 1--62.

\bibitem[{Li et~al.(2021)Li, Chen, Liu, Lu, and Xu}]{li2021rate}
Li, Z.; Chen, P.-Y.; Liu, S.; Lu, S.; and Xu, Y. 2021.
\newblock Rate-improved inexact augmented Lagrangian method for constrained
  nonconvex optimization.
\newblock In \emph{International Conference on Artificial Intelligence and
  Statistics}, 2170--2178. PMLR.

\bibitem[{Li and Xu(2021)}]{doi:10.1287/ijoo.2021.0052}
Li, Z.; and Xu, Y. 2021.
\newblock Augmented Lagrangian–Based First-Order Methods for
  Convex-Constrained Programs with Weakly Convex Objective.
\newblock \emph{INFORMS Journal on Optimization}, 3(4): 373--397.

\bibitem[{Lian et~al.(2016)Lian, Zhang, Hsieh, Huang, and
  Liu}]{lian2016comprehensive}
Lian, X.; Zhang, H.; Hsieh, C.-J.; Huang, Y.; and Liu, J. 2016.
\newblock A comprehensive linear speedup analysis for asynchronous stochastic
  parallel optimization from zeroth-order to first-order.
\newblock \emph{arXiv preprint arXiv:1606.00498}.

\bibitem[{Lin, Lu, and Xiao(2014)}]{lin2014accelerated}
Lin, Q.; Lu, Z.; and Xiao, L. 2014.
\newblock An accelerated proximal coordinate gradient method.
\newblock In \emph{Advances in Neural Information Processing Systems},
  3059--3067.

\bibitem[{Lin, Ma, and Xu(2019)}]{lin2019inexact-pp}
Lin, Q.; Ma, R.; and Xu, Y. 2019.
\newblock Inexact Proximal-Point Penalty Methods for Non-Convex Optimization
  with Non-Convex Constraints.
\newblock \emph{arXiv preprint arXiv:1908.11518}.

\bibitem[{Liu et~al.(2018)Liu, Chen, Chen, and Hero}]{liu2018zeroth}
Liu, S.; Chen, J.; Chen, P.-Y.; and Hero, A. 2018.
\newblock Zeroth-order online alternating direction method of multipliers:
  Convergence analysis and applications.
\newblock In \emph{International Conference on Artificial Intelligence and
  Statistics}, 288--297. PMLR.

\bibitem[{Liu et~al.(2020{\natexlab{a}})Liu, Chen, Kailkhura, Zhang, Hero, and
  Varshney}]{liu2020primer}
Liu, S.; Chen, P.-Y.; Kailkhura, B.; Zhang, G.; Hero, A.; and Varshney, P.~K.
  2020{\natexlab{a}}.
\newblock A Primer on Zeroth-Order Optimization in Signal Processing and
  Machine Learning.
\newblock \emph{IEEE Signal Processing Magazine}.

\bibitem[{Liu et~al.(2016)Liu, Chepuri, Fardad, Ma{\c{s}}azade, Leus, and
  Varshney}]{liu2016sensor}
Liu, S.; Chepuri, S.~P.; Fardad, M.; Ma{\c{s}}azade, E.; Leus, G.; and
  Varshney, P.~K. 2016.
\newblock Sensor selection for estimation with correlated measurement noise.
\newblock \emph{IEEE Transactions on Signal Processing}, 64(13): 3509--3522.

\bibitem[{Liu et~al.(2020{\natexlab{b}})Liu, Lu, Chen, Feng, Xu, Al-Dujaili,
  Hong, and O’Reilly}]{liu2020min}
Liu, S.; Lu, S.; Chen, X.; Feng, Y.; Xu, K.; Al-Dujaili, A.; Hong, M.; and
  O’Reilly, U.-M. 2020{\natexlab{b}}.
\newblock Min-max optimization without gradients: Convergence and applications
  to black-box evasion and poisoning attacks.
\newblock In \emph{International Conference on Machine Learning}, 6282--6293.
  PMLR.

\bibitem[{Ma, Lin, and Yang(2020)}]{ma2020quadratically}
Ma, R.; Lin, Q.; and Yang, T. 2020.
\newblock Quadratically Regularized Subgradient Methods for Weakly Convex
  Optimization with Weakly Convex Constraints.
\newblock In \emph{International Conference on Machine Learning}, 6554--6564.
  PMLR.

\bibitem[{Melo, Monteiro, and Wang(2020{\natexlab{a}})}]{meloiteration2020}
Melo, J.~G.; Monteiro, R.~D.; and Wang, H. 2020{\natexlab{a}}.
\newblock Iteration-complexity of an inexact proximal accelerated augmented
  Lagrangian method for solving linearly constrained smooth nonconvex composite
  optimization problems.
\newblock \emph{Optimization Online}.

\bibitem[{Melo, Monteiro, and Wang(2020{\natexlab{b}})}]{melo2020iteration}
Melo, J.~G.; Monteiro, R.~D.; and Wang, H. 2020{\natexlab{b}}.
\newblock Iteration-complexity of an inexact proximal accelerated augmented
  Lagrangian method for solving linearly constrained smooth nonconvex composite
  optimization problems.
\newblock \emph{arXiv preprint arXiv:2006.08048}.

\bibitem[{Necoara and Nedelcu(2014)}]{necoara2014rate}
Necoara, I.; and Nedelcu, V. 2014.
\newblock Rate analysis of inexact dual first-order methods application to dual
  decomposition.
\newblock \emph{IEEE Transactions on Automatic Control}, 59(5): 1232--1243.

\bibitem[{Nedelcu, Necoara, and Tran-Dinh(2014)}]{nedelcu2014computational}
Nedelcu, V.; Necoara, I.; and Tran-Dinh, Q. 2014.
\newblock Computational complexity of inexact gradient augmented Lagrangian
  methods: application to constrained MPC.
\newblock \emph{SIAM Journal on Control and Optimization}, 52(5): 3109--3134.

\bibitem[{Nesterov and Spokoiny(2017)}]{nesterov2017random}
Nesterov, Y.; and Spokoiny, V. 2017.
\newblock Random gradient-free minimization of convex functions.
\newblock \emph{Foundations of Computational Mathematics}, 17(2): 527--566.

\bibitem[{Ouyang et~al.(2015)Ouyang, Chen, Lan, and
  Pasiliao~Jr}]{ouyang2015accelerated}
Ouyang, Y.; Chen, Y.; Lan, G.; and Pasiliao~Jr, E. 2015.
\newblock An accelerated linearized alternating direction method of
  multipliers.
\newblock \emph{SIAM Journal on Imaging Sciences}, 8(1): 644--681.

\bibitem[{Petricoin~III et~al.(2002)Petricoin~III, Ardekani, Hitt, Levine,
  Fusaro, Steinberg, Mills, Simone, Fishman, Kohn et~al.}]{petricoin2002use}
Petricoin~III, E.~F.; Ardekani, A.~M.; Hitt, B.~A.; Levine, P.~J.; Fusaro,
  V.~A.; Steinberg, S.~M.; Mills, G.~B.; Simone, C.; Fishman, D.~A.; Kohn,
  E.~C.; et~al. 2002.
\newblock Use of proteomic patterns in serum to identify ovarian cancer.
\newblock \emph{The lancet}, 359(9306): 572--577.

\bibitem[{Sahin et~al.(2019)Sahin, Alacaoglu, Latorre, Cevher
  et~al.}]{sahin2019inexact}
Sahin, M.~F.; Alacaoglu, A.; Latorre, F.; Cevher, V.; et~al. 2019.
\newblock An inexact augmented lagrangian framework for nonconvex optimization
  with nonlinear constraints.
\newblock In \emph{Advances in Neural Information Processing Systems},
  13943--13955.

\bibitem[{Xu(2021{\natexlab{a}})}]{xu2021first}
Xu, Y. 2021{\natexlab{a}}.
\newblock First-order methods for constrained convex programming based on
  linearized augmented Lagrangian function.
\newblock \emph{Informs Journal on Optimization}, 3(1): 89--117.

\bibitem[{Xu(2021{\natexlab{b}})}]{xu2021iteration}
Xu, Y. 2021{\natexlab{b}}.
\newblock Iteration complexity of inexact augmented lagrangian methods for
  constrained convex programming.
\newblock \emph{Mathematical Programming}, 185(1): 199--244.

\bibitem[{Zhang and Luo(2020)}]{zhang2020global}
Zhang, J.; and Luo, Z. 2020.
\newblock A global dual error bound and its application to the analysis of
  linearly constrained nonconvex optimization.
\newblock \emph{arXiv preprint arXiv:2006.16440}.

\end{thebibliography}
\appendix

\section{Proofs}
In this section, we provide detailed proofs of our theorems.
\subsection{Proof of Theorem~\ref{thm:coor_grad_coef}}
Denote $m = \frac{p}{2}$. Note that the constants $C_1, \cdots, C_m$ in Theorem~\ref{thm:coor_grad_coef} satisfy the following $m$ equalities:
\begin{equation}\label{eq:const_lin_sys}
\begin{aligned} 
C_1 + 2C_2 + \cdots + mC_m &= \frac{1}{2a},  \\
C_1 + 2^3 C_2 + \cdots + m^3 C_m &= 0,  \\
& \vdots  \\
C_1 + 2^{2m-1} C_2 + \cdots + m^{2m-1} C_m &= 0.
\end{aligned}
\end{equation}
Since $f$ is $M_j$ coordinate $j$-smooth, by plugging $b \in \{ ma, \cdots, a, -a, \cdots, -ma \}$ into Lemma \ref{lemma:smooth_ineq}, we have the following $2m$ inequalities:
\begin{align*}
\textstyle -\frac{|C_m|M_j}{(j+1)!}(ma)^{j+1} \le & \ C_m \Big( f(\vx+mae_i) - f(\vx) \\ 
& - ma\nabla_i f(\vx) - \cdots - \frac{m^j a^j}{j!} \nabla_i^j f(\vx) \Big) \\ 
\le & \ \frac{|C_m|M_j}{(j+1)!}(ma)^{j+1}, \\
& \vdots \\
\textstyle -\frac{|C_1|M_j}{(j+1)!}a^{j+1} \le & \ C_1 \Big( f(\vx+ae_i) - f(\vx) - a\nabla_i f(\vx)\\ 
& - \cdots - \frac{ a^j}{j!} \nabla_i^j f(\vx) \Big) \\ 
\le & \ \frac{|C_1|M_j}{(j+1)!}a^{j+1}, \\
\textstyle -\frac{|C_1|M_j}{(j+1)!}a^{j+1} \le & \ -C_1 \Big( f(\vx-ae_i) - f(\vx) + a\nabla_i f(\vx)\\ 
& - \cdots - \frac{ (-1)^j a^j}{j!} \nabla_i^j f(\vx) \Big)\\ 
\le & \ \frac{|C_1|M_j}{(j+1)!}a^{j+1}, \\
& \vdots \\
\textstyle -\frac{|C_m|M_j}{(j+1)!}(ma)^{j+1} \le & \ -C_m \Big( f(\vx-mae_i) - f(\vx) \\ 
& \hspace{-1cm} + ma\nabla_i f(\vx) - \cdots   - \frac{(-1)^j m^j a^j}{j!} \nabla_i^j f(\vx) \Big)\\ 
\le & \ \frac{|C_m|M_j}{(j+1)!}(ma)^{j+1}.
\end{align*}
Summing up the above $2m$ inequalities, we have
\begin{align*}
& \Big| \tilde{\nabla}_i f(\vx)  -  (C_1+2C_2+\cdots+mC_m)2a \nabla_i f(\vx) \\ 
& - (C_1+2^3 C_2+ \cdots + m^3 C_m)\frac{2a^3}{3!} \nabla_i^{3} f(\vx) - \\
& \cdots - (C_1+2^{2m-1}C_2+\cdots+m^{2m-1}C_m) \\ 
& \frac{2a^{2m-1}}{(2m-1)!}\nabla_i^{2m-1} f(\vx) \Big| \\ 
\le & \ \sum_{q=1}^{m} |C_q| \frac{M_j}{(j+1)!}(qa)^{j+1}.
\end{align*}
The above equality combined with \eqref{eq:const_lin_sys} gives us
\begin{equation*}
|\tilde{\nabla}_i f(\vx) - \nabla_i f(\vx)| \le \sum_{q=1}^{m} |C_q| \frac{M_j q^{j+1}}{(j+1)!}a^{j+1},
\end{equation*}
which is exactly \eqref{eq:grad_err_bd}.

\subsection{Proof of Theorem~\ref{thm:apcg_cplx_expected}}
To prove Theorem~\ref{thm:apcg_cplx_expected}, first we bound the objective error in Theorem \ref{thm:apcg_conv} below, next we bound the stationarity gap by the objective error in Theorem \ref{thm:subdiff_obj} below, then we combine these two theorems with the assumed parameter settings and get the desired results. Below, we present the detailed proof of Theorem~\ref{thm:apcg_cplx_expected}.\\

Denote $$\tilde{\partial} F(\vx) = \tilde{\nabla} G(\vx) + \partial H(\vx).$$ By the updates of Algorithm \ref{alg:apcg}, following the proof of 
Lemmas 2 and 3 in \cite{lin2014accelerated}, we immediately have the following two lemmas.

\begin{lemma}\label{apcg_lemma_2}
Let $\{\vx^k\}$ be generated from Algorithm~\ref{alg:apcg}. Then we have $\vx^k = \sum_{l=0}^{k} \theta_l^k \vz^l$, where $\theta_0^0 = 1, \theta_0^1 = 1-\sqrt{\mu}, \theta_1^1 = \sqrt{\frac{\mu}{L}}$, and $\forall k \ge 1$,\\
$\theta_l^{k+1}=
\begin{cases}
\sqrt{\mu}, \text{ if } l = k+1,\\
(1-\frac{\mu}{dL})\frac{(d+1)\alpha}{\alpha+1}-\frac{(1-\alpha)\mu}{dL\alpha}, \text{ if } l = k,\\
(1-\frac{\mu}{dL})\frac{1}{\alpha+1}\theta_l^k, \text{ if } l = 0,\dots,k-1.
\end{cases}
$
\end{lemma}

\begin{lemma}\label{apcg_lemma_3} 
Let $\hat{\psi}_k := \sum_{l=0}^{k} \theta_l^k H(\vz^l)$. Then $\forall k \ge 0$, we have $H(\vx^k) \le \hat{\psi}_k$ and 
$$\EE_{i_k}[\hat{\psi}_{k+1}] \le \alpha H(\tilde{\vz}^{k+1})+(1-\alpha) \hat{\psi}_k,$$ 
where 
\begin{align}\label{apcg_eq11}
\tilde{\vz}^{k+1} := \argmin_{\vx \in \RR^d} \Big\{ \frac{d\alpha L}{2} \Vert \vx-(1-\alpha)\vz^k-\alpha \vy^k \Vert^2 \nonumber \\ 
+\langle \tilde{\nabla}G(\vy^k), \vx-\vy^k \rangle + H(\vx) \Big\} .
\end{align}
\end{lemma}

Theorem~\ref{thm:apcg_conv} below extends from Theorem 1 in \cite{lin2014accelerated}, and gives the convergence rate of Algorithm~\ref{alg:apcg}.
\begin{theorem}[ZO-APCG convergence rate]\label{thm:apcg_conv}
Let $\{ \vx^t \}_{t=0}^{K}$ be generated from Algorithm \ref{alg:apcg}. Then
\begin{align}\label{eq:apcg_conv}
\EE [F(\vx^K)]-F^* \le \Big( 1-\frac{1}{d}\sqrt{\frac{\mu}{L}} \Big)^K \Big(F(\vx^0)-F^* \nonumber \\ 
+ \frac{\mu}{2}\Vert \vx^0-\vx^* \Vert^2 \Big) + ED + \sum_{i=1}^{d} E_i D_i.
\end{align}
\end{theorem}

\begin{proof}
By updates of $\vz^{k+1}$ and $\vx^{k+1}$ in Algorithm \ref{alg:apcg}, 
\begin{equation}\label{apcg_eq13}
    x^{k+1}_i =
\begin{cases}
y_i^k+d\alpha(z^{k+1}_i-z^k_i)+\frac{\mu}{dL}(z_i^k-y_i^k), \text{ if } i = i_k\\
y_i^k, \text{ if } i \neq i_k.
\end{cases}
\end{equation}
By the update of $\vy^k$,
\begin{equation}\label{apcg_eq32}
\vz^k-\vy^k = -\frac{1}{\alpha}(\vx^k-\vy^k).
\end{equation} 
Also by the update of $\vx^{k+1}$ and $\alpha = \frac{1}{d}\sqrt{\frac{\mu}{L}}$, we have
\begin{align*}
\vx^{k+1}-\vy^k = & \ \sqrt{\frac{\mu}{L}} \vz^{k+1}-(1-\alpha)\sqrt{\frac{\mu}{L}}\vz^k-\frac{\mu}{dL}\vy^k\\ 
= & \ \sqrt{\frac{\mu}{L}}\vz^{k+1}-(1-\alpha)\sqrt{\frac{\mu}{L}}(\vz^k-\vy^k) \\ 
& - \left( (1-\alpha)\sqrt{\frac{\mu}{L}}+\frac{\mu}{dL} \right) \vy^k,
\end{align*} 
which combined with \eqref{apcg_eq32} gives us $$\vx^{k+1}-\vy^k = d\big(\alpha(\vz^{k+1}-\vy^k)+(1-\alpha)(\vx^k-\vy^k)\big).$$
Combining the above equation with \eqref{apcg_eq13} and $L$ smoothness of $G$, we have
\begin{align*}
    G(\vx^{k+1}) \le & \ G(\vy^k)+ \nabla_{i_k} G(\vy^k) (x^{k+1}_{i_k}-y^k_{i_k}) \\ 
    & + \frac{L}{2} \Vert x^{k+1}_{i_k}-y^k_{i_k} \Vert^2\\
    \le & \ (1-\alpha) (G(\vy^k) + d \nabla_{i_k} G(\vy^k) (x^{k}_{i_k}-y^k_{i_k}) ) \\ 
    & + \alpha (G(\vy^k)+d \tilde{\nabla}_{i_k} G(\vy^k) (z^{k+1}_{i_k}-y^k_{i_k}) ) \\
    & + \frac{d^2 L}{2} [\alpha (\vz^{k+1}-\vy^k)+(1-\alpha)(\vx^k-\vy^k)]_{i_k} ^2 \\ 
    & + \alpha d E_{i_k}D_{i_k}.
\end{align*}
Thus by $\mu$-strong convexity of $G$, the choice of $i_k$, and the definition of $\tilde{\vz}^{k+1}$, it holds
\begin{align}\label{apcg_eq33} 
\EE_{i_k}[G(\vx^{k+1})] \le & \ (1-\alpha) G(\vx^k)+\alpha \Big( G(\vy^k) + \langle \tilde{\nabla} G(\vy^k), \nonumber \\ 
& \tilde{\vz}^{k+1}-\vy^k \rangle \Big) + \frac{d L}{2}\Vert \alpha (\tilde{\vz}^{k+1}-\vy^k) \nonumber \\
& + (1-\alpha)(\vx^k-\vy^k) \Vert^2  + \alpha \sum_{i=1}^{n} E_{i}D_{i}.
\end{align}
In addition, by \eqref{apcg_eq32}, 
\begin{align}\label{apcg_eq34} 
\frac{dL}{2} \Vert \alpha (\tilde{\vz}^{k+1}-\vy^k)+(1-\alpha)(\vx^k-\vy^k) \Vert^2 \nonumber \\ 
= \frac{\mu }{2d} \Vert \tilde{\vz}^{k+1}-(1-\alpha)\vz^k-\alpha \vy^k \Vert^2.
\end{align}
Combining the above equality with \eqref{apcg_eq33} and $\alpha = \frac{1}{d}\sqrt{\frac{\mu}{L}}$, we have 
\begin{align*}
&~\EE_{i_k}[G(\vx^{k+1})] \\
\le &~ (1-\alpha) G(\vx^k)+\alpha \Big[G(\vy^k) +\langle \tilde{\nabla} G(\vy^k), \tilde{\vz}^{k+1}-\vy^k \rangle \\
& ~+ \frac{\sqrt{\mu L}}{2}\Vert \tilde{\vz}^{k+1}-(1-\alpha)\vz^k-\alpha \vy^k \Vert^2\Big] + \alpha \sum_{i=1}^{d}E_i D_i,
\end{align*}
which combined with Lemma \ref{apcg_lemma_3} gives
\begin{align}\label{apcg_eq35}
    \EE_{i_k}[G(\vx^{k+1})+\hat{\psi}_{k+1}] \le & \ (1-\alpha)(G(\vx^k)+\hat{\psi}_k) \nonumber \\ 
    & +\alpha V(\tilde{\vz}^{k+1}) + \alpha \sum_{i=1}^{d} E_{i}D_{i}.
\end{align}
In the above
\begin{align*}
V(\vx) := & ~G(\vy^k) + \langle \tilde{\nabla}G(\vy^k),\vx-\vy^k \rangle \\
& ~+ \frac{\sqrt{\mu L}}{2} \Vert \vx-(1-\alpha)\vz^k-\alpha \vy^k \Vert^2+H(\vx).
\end{align*}
 By \eqref{apcg_eq11},
\begin{equation}\label{apcg_eq36}
    \tilde{\vz}^{k+1} = \argmin_{\vx \in \RR^d} V(\vx).
\end{equation}
Note $V$ is $\sqrt{\mu L}$-strongly convex, so by \eqref{apcg_eq36}, $V(\vx^*) \ge V(\tilde{\vz}^{k+1})+\frac{\sqrt{\mu L}}{2}\Vert \vx^*-\tilde{\vz}^{k+1} \Vert^2$. Thus, 
\begin{align*}
    V(\tilde{\vz}^{k+1}) \le & \ V(\vx^*) - \frac{\sqrt{\mu L}}{2}\Vert \vx^*-\tilde{\vz}^{k+1} \Vert^2\\
    = & \ G(\vy^k) + \langle \tilde{\nabla} G(\vy^k),\vx^*-\vy^k \rangle + \frac{\sqrt{\mu L}}{2}\Vert \vx^* \\ 
    & - (1-\alpha)\vz^k - \alpha \vy^k \Vert^2 + H(\vx^*)\\
     &- \frac{\sqrt{\mu L}}{2}\Vert \vx^*-\tilde{\vz}^{k+1} \Vert^2\\
    \le & \ G(\vy^k) + \langle \nabla G(\vy^k),\vx^*-\vy^k \rangle + \frac{\sqrt{\mu L}}{2}\Vert \vx^* \\ 
    & - (1-\alpha)\vz^k - \alpha \vy^k \Vert^2 + H(\vx^*)\\ 
    & - \frac{\sqrt{\mu L}}{2}\Vert \vx^*-\tilde{\vz}^{k+1} \Vert^2 + ED\\ 
    \le & \ G(\vx^*) - \frac{\mu}{2} \Vert \vx^*-\vy^k \Vert^2 + \frac{\sqrt{\mu L}}{2} \Vert \vx^* \\ 
    & -(1-\alpha)\vz^k -\alpha \vy^k \Vert^2 + H(\vx^*), \\
    & - \frac{\sqrt{\mu L}}{2} \Vert \vx^*-\tilde{\vz}^{k+1} \Vert^2 + ED,
\end{align*}
where the last inequality holds by $\mu$-strong convexity of $G$.
Combining the last inequality with \eqref{apcg_eq35}, we have
\begin{equation}\label{apcg_eq37}
\begin{aligned}
   &~ \EE_{i_k} [G(\vx^{k+1})+\hat{\psi}_{k+1}] \\
   \le &~ (1-\alpha)(G(\vx^k)+\hat{\psi}_k)+\alpha F^*  - \frac{\alpha \mu}{2} \Vert \vx^*-\vy^k \Vert^2   \\ 
    & \hspace{-0.5cm}- \frac{\mu}{2d} \Vert \vx^* - \tilde{\vz}^{k+1} \Vert^2 + \frac{\mu}{2d} \Vert \vx^*-(1-\alpha)\vz^k-\alpha \vy^k \Vert^2 \\ 
    & + \alpha E D + \alpha \sum_{i=1}^{d}E_i D_i. 
\end{aligned}
\end{equation}
Now, by the convexity of $\Vert \cdot \Vert^2$, it holds
\begin{equation}\label{apcg_eq38}
\begin{aligned}
&~    \Vert \vx^*-(1-\alpha)\vz^k-\alpha \vy^k \Vert^2\\
\le & ~ (1-\alpha) \Vert \vx^*-\vz^k \Vert^2 + \alpha \Vert \vx^*-\vy^k \Vert^2.
    \end{aligned}
\end{equation}
Note from the updates of Algorithm \ref{alg:apcg},
\begin{equation}\label{apcg_eq12}
    \vz^{k+1}_i = 
    \begin{cases}
    \tilde{\vz}^{k+1}_i, \text{ if } i = i_k,\\
    (1-\alpha)\vz^k_i + \alpha \vy^k_i, \text{ if }i \neq i_k,
    \end{cases}
\end{equation}
which implies 
\begin{align*}
    & \EE_{i_k}[\frac{\mu}{2}\Vert \vx^*-\vz^{k+1} \Vert^2]\\ 
    = & \ \frac{\mu}{2} \left[ \frac{d-1}{d}\Vert \vx^*-(1-\alpha)\vz^k-\alpha \vy^k \Vert^2 + \frac{1}{d} \Vert \vx^*-\tilde{\vz}^{k+1} \Vert^2 \right]\\
    = & \ \frac{\mu(d-1)}{2d} \Vert \vx^*-(1-\alpha)\vz^k-\alpha \vy^k \Vert^2 + \frac{\mu}{2d} \Vert \vx^*-\tilde{\vz}^{k+1} \Vert^2\\
    = & \ \frac{\mu}{2} \Vert \vx^*-(1-\alpha)\vz^k-\alpha \vy^k \Vert^2 - \frac{\mu}{2d} \Vert \vx^*-(1-\alpha)\vz^k \\ 
    & - \alpha \vy^k \Vert^2 + \frac{\mu}{2d} \Vert \vx^*-\tilde{\vz}^{k+1} \Vert^2\\
    \le & \ \frac{(1-\alpha)\mu}{2} \Vert \vx^*-\vz^k \Vert^2 + \frac{\alpha \mu}{2} \Vert \vx^*-\vy^k \Vert^2  \\ 
    & - \frac{\mu}{2d} \Vert \vx^*-(1-\alpha)\vz^k -\alpha \vy^k \Vert^2 + \frac{\mu}{2d} \Vert \vx^*-\tilde{\vz}^{k+1} \Vert^2,
\end{align*}
where the last inequality follows from \eqref{apcg_eq38}.
Combining the last inequality with \eqref{apcg_eq37}, we get
\begin{align*}
    & \EE_{i_k} \left[ G(\vx^{k+1})+\hat{\psi}_{k+1}+\frac{\mu}{2}\Vert \vx^*-\vz^{k+1} \Vert^2 \right]\\ 
    \le & \ (1-\alpha)(G(\vx^k)+\hat{\psi}_k+\frac{\mu}{2}\Vert \vx^*-\vz^k \Vert^2) + \alpha F^*\\ 
    &+ \alpha E D + \alpha \sum_{i=1}^{d} E_i D_i,
\end{align*}
which implies
\begin{align*}
    & \EE_{i_k} \left[ G(\vx^{k+1})+\hat{\psi}_{k+1} - F^* +\frac{\mu}{2}\Vert \vx^*-\vz^{k+1} \Vert^2 \right] \\ 
    \le & \ (1-\alpha)(G(\vx^k)+\hat{\psi}_k - F^* +\frac{\mu}{2}\Vert \vx^*-\vz^k \Vert^2) \\
    & + \alpha E D + \alpha \sum_{i=1}^{d} E_i D_i.
\end{align*}
Thus,
\begin{align*}
    &\EE \left[ G(\vx^k) + \hat{\psi}_k - F^* + \frac{\mu}{2} \Vert \vx^*-\vz^k \Vert^2 \right] \\
    \le & \ (1-\alpha)^k \left[ F(\vx^0)-F^*+\frac{\mu}{2}\Vert \vx^*-\vx^0 \Vert^2 \right] \\ 
    & + \left( \alpha E D + \alpha \sum_{i=1}^{d} E_i D_i \right) \sum_{t=0}^{k-1}(1-\alpha)^t.\\
\end{align*}
Hence,
\begin{align*}
    &\EE \left[ F(\vx^k) - F^* + \frac{\mu}{2} \Vert \vx^*-\vz^k \Vert^2 \right] \le (1-\alpha)^k [F(\vx^0)-F^* \\
    & +\frac{\mu}{2}\Vert \vx^*-\vx^0 \Vert^2] + E D + \sum_{i=1}^{d} E_i D_i.
\end{align*}
where the inequality holds by $F(\vx^k) \le G(\vx^k) + \hat{\psi}_k$, $\vx^k = \sum_{l=0}^{k} \theta_l^k \vz^l$, and the definition of $\hat{\psi_k}$.
\end{proof}

Theorem~\ref{thm:subdiff_obj} below bounds the subdifferential by the objective error.
\begin{theorem}\label{thm:subdiff_obj}
Let 
\begin{equation}\label{eq:postproc0} 
\hat{\vx} = \argmin_{\vx' \in \RR^d} \langle \tilde{\nabla} G(\vx), \vx'-\vx \rangle + \frac{L}{2}\Vert \vx'-\vx \Vert^2 + H(\vx')  , 
\end{equation}
as in the postprocessing step of Algorithm \ref{alg:apcg}, where $\Vert \tilde{\nabla} G(\vx) - \nabla G(\vx) \Vert \le E$.
Then
\begin{equation}\label{eq:subdiff_obj}
\dist(0, \partial F(\hat{\vx})) \le 4L\sqrt{\frac{2(F(\vx)-F^*)}{\mu}} + 2L\sqrt{\frac{2ED}{\mu}}+E.
\end{equation}
\end{theorem}

\begin{proof}
First, observe that
\begin{align}\label{eq:38.5}
F(\hat{\vx}) \le & \ G(\vx) + \langle \nabla G(\vx), \hat{\vx}-\vx \rangle + \frac{L}{2} \Vert \hat{\vx}-\vx \Vert^2 + H(\hat{\vx}) \nonumber\\
\le & \ G(\vx) + \langle \tilde{\nabla} G(\vx), \hat{\vx}-\vx \rangle + \frac{L}{2} \Vert \hat{\vx}-\vx \Vert^2 + H(\hat{\vx}) \nonumber\\ 
& + ED \nonumber\\
\le & \ G(\vx) + H(\vx) + ED \nonumber\\
= & \ F(\vx) + ED,
\end{align}
where in above, the first inequality follows from $L$ smoothness of $G$, and the third inequality follows from \eqref{eq:postproc0}.

Then, by the $\mu$-strong convexity of $G$, we have
\begin{equation}\label{eq:iterates_obj}
\frac{\mu}{2}\Vert \vx'-\vx^* \Vert^2 \le F(\vx')-F^*, \forall \vx' \in \RR^d.
\end{equation}
Furthermore, by \eqref{eq:postproc0}, we have
\begin{equation}\label{eq:postproc}
\vzero \in \tilde{\nabla} G(\vx) + L(\hat{\vx}-\vx) + \partial H(\hat{\vx}).
\end{equation}
Thus,
\begin{align*}
&\dist(0, \partial F(\hat{\vx}))\\ 
\le & \ \Vert \nabla G(\hat{\vx})-\nabla G(\vx) + \nabla G(\vx) - \tilde{\nabla}G(\vx)-L(\hat{\vx}-\vx) \Vert\\
\le & \ \Vert \nabla G(\hat{\vx})-\nabla G(\vx) \Vert + \Vert \nabla G(\vx) - \tilde{\nabla}G(\vx) \Vert \\ 
&+ \Vert L(\hat{\vx}-\vx) \Vert \\
\le & \ 2L \Vert \hat{\vx}-\vx \Vert + E\\
\le & \ 2L(\Vert \hat{\vx}-\vx^* \Vert + \Vert \vx-\vx^* \Vert) + E\\
\le & \ 2L\sqrt{\frac{2}{\mu}}(\sqrt{F(\hat{\vx})-F^*}+\sqrt{F(\vx)-F^*}) + E \\
\le & \ 4L\sqrt{\frac{2(F(\vx)-F^*)}{\mu}} + 2L\sqrt{\frac{2ED}{\mu}} + E
\end{align*}
where in above, the first inequality follows from \eqref{eq:postproc}, the third inequality follows from $L$ smoothness of $G$, the fifth inequality follows from \eqref{eq:iterates_obj}, and the last inequality uses \eqref{eq:38.5}.
\end{proof}

Based on Theorem~\ref{thm:apcg_conv} and Theorem~\ref{thm:subdiff_obj} above, now we are ready to prove Theorem~\ref{thm:apcg_cplx_expected}. 

By Theorem~\ref{thm:apcg_conv}, \eqref{eq:small_grad_err}, and the definition of $T$, we have 
$$\EE[F(\vx^T)] - F^* \le \frac{\bar{\vareps}}{2}+\frac{\bar{\vareps}}{2} = \bar{\vareps}, $$
where $\bar{\vareps} = \frac{\mu}{512L^2}\vareps^2$. Combining above inequality with Theorem~\ref{thm:subdiff_obj} and \eqref{eq:small_grad_err}, we have
\begin{align*}
&\EE[ \dist(\vzero, \partial F(\hat{\vx}^T)) ] \\ 
&\le 4L\sqrt{\frac{2(F(\vx^T)-F^*)}{\mu}} + 2L\sqrt{\frac{2ED}{\mu}} + E \\
&\le \frac{\vareps}{4}+\frac{\vareps}{4} = \frac{\vareps}{2}.
\end{align*}
Thus, 
$$\EE[ \dist(\vzero, \tilde{\partial}F(\hat{\vx}^T)) ] \le \EE[ \dist(\vzero, \partial F(\hat{\vx}^T)) ] + E \le \frac{3\vareps}{4},$$
and Algorithm \ref{alg:apcg} must stop within $T$ iterations.
This completes the proof of Theorem~\ref{thm:apcg_cplx_expected}.
\subsection{Proof of Theorem~\ref{thm:exp-cplx}} 
Observe that 
\begin{align*}
\EE[K(\vareps)] = & ~ \sum_{k=1}^{\infty}k P(K(\vareps)=k)\\
\leq & ~ t+\sum_{k=t+1}^{\infty}kP(K(\vareps)=k),\forall t \in \ZZ^+.
\end{align*}
Note 
\begin{align*}
P(K(\vareps)=k)= & ~P(q_1>\vareps, \dots, q_{k-1}>\vareps, q_k \leq \vareps) \\
\leq & ~ P(q_{k-1}>\vareps) \leq \frac{\EE[q_{k-1}]}{\vareps} \leq \frac{C \eta^{k-1}}{\vareps}.
\end{align*}
Thus, 
\begin{align}\label{eq:E1}
\EE[K(\vareps)] \leq & ~ t+\sum_{k=t}^{\infty}(k+1)\frac{C\eta^k}{\vareps}\nonumber \\ 
= & ~ t+\frac{C}{\vareps} \sum_{k=t}^{\infty} (k+1) \eta^k \nonumber \\ 
= & ~ t+\frac{C}{\vareps}(\frac{\eta^t}{1-\eta} + \sum_{k=t}^{\infty}k\eta^k).
\end{align}
Let $S_t = \sum_{k=t}^{\infty} k \eta^k$. So $\eta S_t = \sum_{k=t}^{\infty} k\eta^{k+1}$, and
\begin{align*}
 S_t - \eta S_t &= \sum_{k=t}^{\infty} k\eta^k - \sum_{k=t}^{\infty} k\eta^{k+1} = t\eta^t + \sum_{k=t+1}^{\infty}\eta^k \\
 &= t\eta^t + \frac{\eta^{t+1}}{1-\eta}. 
\end{align*}
Thus $S_t = \frac{1}{1-\eta} (t\eta^t + \frac{\eta^{t+1}}{1-\eta})$.
Combining the above equation with \eqref{eq:E1}, we have $\forall t \in \ZZ^+$,
\begin{align*}
 \EE[K(\vareps)] &\leq t+\frac{C}{\vareps}(\frac{\eta^t}{1-\eta}+\frac{1}{1-\eta}(t\eta^t+\frac{\eta^{t+1}}{1-\eta})) \\ 
 &= t+\frac{C}{\vareps(1-\eta)}(t+\frac{1}{1-\eta})\eta^t. 
\end{align*}
Let $\psi(t) = t+\frac{C}{\vareps(1-\eta)}(t+\frac{1}{1-\eta})\eta^t$. 
Now we want to choose some $t \in \ZZ^+$ to bound $\psi(t)$ well. 
Here we choose $t = \lceil s \rceil$, where $s = \log_{\frac{1}{\eta}} [\frac{C}{\vareps (1-\eta)^2}]$. So we have $t \in [ s,s+1 )$ and $\eta^t \in (\eta^{s+1}, \eta^s] = ( \frac{\eta \vareps (1-\eta)^2}{C}, \frac{\vareps (1-\eta)^2}{C} ]$. 
Hence,
\begin{align*}
\EE[K(\vareps)] &\leq \psi(t)\\ 
&\leq s+1+\frac{C}{\vareps(1-\eta)} (s+1+\frac{1}{1-\eta}) \frac{\vareps(1-\eta)^2}{C}\\
&= (2-\eta)(s+1)+1\\
&= \frac{2-\eta}{\log \frac{1}{\eta}}\log\frac{C}{\vareps(1-\eta)^2} +3-\eta\\
&\le \frac{2-\eta}{1-\eta}\log\frac{C}{\vareps(1-\eta)^2} +3-\eta.
\end{align*}
\subsection{Proof of Theorem~\ref{thm:ippm-compl}} 
	Let $\Phi_t(\vx) := \Phi(\vx)+\rho \Vert \vx-\vx^t \Vert^2$ and $\Phi_t^*=\min_\vx\Phi_t(\vx)$ for each $t\ge0$. 
	Note we have $\dist(\vzero, \partial \Phi_t(\vx^{t+1})) \le \delta=\frac{\vareps}{4}$, and also $\Phi_t$ is $\rho$-strongly convex. Hence $\Phi_t(\vx^{t+1}) - \Phi_t^* \le \frac{\delta^2}{2\rho}$, and $\Phi(\vx^{t+1})+\rho\Vert\vx^{t+1}-\vx^t\Vert^2-\Phi(\vx^t) \le \frac{\delta^2}{2\rho}$. Thus,
	\begin{align}
	&\Phi(\vx^T)-\Phi(\vx^0)+\rho\sum_{t=0}^{T-1}\Vert \vx^{t+1}-\vx^t \Vert^2 \le \frac{T\delta^2}{2\rho} \nonumber\\
	&T\min_{0 \le t \le T-1} \Vert \vx^{t+1}-\vx^t \Vert^2 \le \frac{1}{\rho} \left(\frac{T \delta^2}{2\rho}+[\Phi(\vx^0)-\Phi(\vx^T)]\right) \nonumber\\
	&2\rho \min_{0 \le t \le T-1} \Vert \vx^{t+1}-\vx^t \Vert \le 2\sqrt{\frac{\delta^2}{2}+\frac{\rho [\Phi(\vx^0)-\Phi^*] }{T}}. \label{eq:0}
	\end{align}
	Since $T \ge \frac{32\rho}{\vareps^2}[\Phi(\vx^0)-\Phi^*]$ and $\delta = \frac{\vareps}{4}$, we have
	\begin{equation} \label{eq:2}
	\frac{\rho}{T}[\Phi(\vx^0)-\Phi^*] \le \frac{\vareps^2}{32},
	\end{equation}
	and thus \eqref{eq:0} implies
	\begin{equation}\label{2}
	2\rho \min_{0 \le t \le T-1} \Vert \vx^{t+1}-\vx^t \Vert \le \frac{\vareps}{2}.
	\end{equation}
	Therefore, the ZO-iPPM subroutine in Algorithm~\ref{alg:ialm} must stop within $T$ iterations, from its stopping condition, and when it stops, the output $\vx^S$ satisfies $2\rho \Vert \vx^{S}-\vx^{S-1} \Vert \le \frac{\vareps}{2}$.
	
	Now recall $\dist(0, \partial \Phi_t(\vx^{t+1})) \le \delta = \frac{\vareps}{4}$, i.e.,
	\begin{equation}\label{1}
	\dist(0, \partial \Phi(\vx^{t+1}) + 2\rho(\vx^{t+1}-\vx^t)) \le \frac{\vareps}{4}, \forall t \ge 0.
	\end{equation}
	The above inequality together with $2\rho \Vert \vx^{S}-\vx^{S-1} \Vert \le \frac{\vareps}{2}$ gives
	\begin{equation*}
	\dist(\vzero, \partial \Phi(\vx^{S})) \le \vareps,
	\end{equation*}
	which implies that $\vx^{S}$ is an $\vareps$-stationary point to \eqref{eq:nc_prob}.
	
	Finally, we apply Corollary~\ref{cor:apcg_cplx} to obtain the expected overall complexity and complete the proof.

\subsection{Proof of Theorem~\ref{thm:ialm-cplx-2}} 
To prove Theorem~\ref{thm:ialm-cplx-2}, first we bound the dual variable $\Vert \vy^k \Vert$, next we establish upper and lower bounds of the AL objective value inside every outer iteration, then we combine above results with Theorem \ref{thm:ippm-compl} and Corollary \ref{cor:apcg_cplx} to show the total query complexity to reach a near-KKT point, finally we establish the improved query complexity in the special case when the constraints are convex. Below, we present the detailed proof of Theorem~\ref{thm:ialm-cplx-2}.\\

	First, by \eqref{eq:alm-y}, $\vy^0 = \vzero$, and the definition of $w_k$ in Theorem~\ref{thm:ialm-cplx-2}, we have
	\begin{align} \label{yk-bound-general}
	\Vert \vy^{k} \Vert &\le \sum_{t=0}^{k-1} w_t \Vert \vc(\vx^{t+1}) \Vert = \sum_{t=0}^{k-1} M(t+1)^q =: y_k \nonumber\\ 
	&= O(k^{q+1}), \forall k \ge 0. 
	\end{align}
	Following the first part of the proof of Theorem 2 in \cite{li2021rate}, we can easily show that at most $K=O(\log \vareps^{-1})$ outer iALM iterations are needed to guarantee $\vx^{K}$ to be an $\vareps$-KKT point of \eqref{eq:ncp}. Hence, $\beta_k = O(\vareps^{-1}), \forall\, 0 \le k \le K$.
	
	Combining the above bound on $K$ with \eqref{yk-bound-general}, we have 
	\begin{align*}
	\Vert \vy^k \Vert &\le y_{K} := \sum_{t=0}^{K-1} M(K+1)^q = O(K^{q+1})\\ 
	&= O\big((\log \vareps^{-1})^{q+1}\big),\, \forall 1 \le k \le K.
	\end{align*}
	Hence from \eqref{eq:hat_rho_k}, we have $\hat{\rho}_k = O(\beta_k) = O(\vareps^{-1}), \hat{L}_k = O(\beta_k) = O(\vareps^{-1}),\, \forall\, 0 \le k \le K$. 
	
	Notice that equations (41) and (42) in \cite{li2021rate} still hold with $y_{\max}$ replaced by $y_k$. Hence,  $\forall k \le K, \forall \vx \in \dom(h),$
	\begin{equation*}
	\cL_{\beta_k}(\vx^k,\vy^k) - \cL_{\beta_k}(\vx,\vy^k) = O\left(y_k\left(1+\frac{y_k}{\beta_k}\right)\right).
	\end{equation*}
	The above equation together with Theorem \ref{thm:ippm-compl} gives that for any $k \le K$, at most $T_k^{\mathrm{PPM}}$ iPPM iterations are needed to terminate the ZO-iPPM subroutine in Algorithm~\ref{alg:ialm} at the $k$-th outer iALM iteration, where
	\begin{align*}
	T_k^{\mathrm{PPM}} &= \left\lceil \frac{32\hat{\rho}_k}{\vareps^2}\big(\cL_{\beta_k}(\vx^k,\vy^k)-\min_{\vx} \cL_{\beta_k}(\vx,\vy^k)\big) \right\rceil\\ 
	&= O \left(\frac{\hat{\rho}_k y_k\left(1+\frac{y_k}{\beta_k}\right)}{\vareps^2} \right).
	\end{align*}
	Also, by Corollary \ref{cor:apcg_cplx}, at most $T_k^{\mathrm{APCU}}$ function value queries are needed to terminate Algorithm \ref{alg:apcg}, where
	\begin{equation*}
	\EE[ T_k^{\mathrm{APCU}} ] = \tilde{O}\left(d\sqrt{\frac{\hat{L}_k}{\hat{\rho_k}}} \right), \forall k \ge 0.
	\end{equation*}
	
	Therefore, for all $k \le K$, 
	\begin{align*}
	\EE[ T_k^{\mathrm{PPM}} T_k^{\mathrm{APG}} ] &= \tilde{O}\left( \frac{d\sqrt{\hat{L}_k \hat{\rho}_k}}{\vareps^2} y_k\left(1+\frac{y_k}{\beta_k}\right) \right) \\
	&= \tilde{O}\left( \frac{d y_k }{\vareps^2} (\beta_k + y_k) \right) \\
	&= \tilde{O}\left( \frac{d k^{q+1} }{\vareps^2} (\sigma^k + k^{q+1}) \right)\\
	& = \tilde{O}\left( \frac{d K^{q+1} }{\vareps^2} (\sigma^K + K^{q+1}) \right)\\
	&= O\left( \frac{d (\log \vareps^{-1})^{q+2}}{\vareps^3} \right)\\
	&= \tilde{O}\left( \frac{d }{\vareps^3} \right),
	\end{align*}
	where the second equation is from $\hat{L}_k = O(\beta_k)$ and $\hat{\rho}_k = O(\beta_k)$, and the fifth one is obtained by $K = O(\log \vareps^{-1})$.
	
	Consequently, for a general nonlinear $\vc(\cdot)$, at most $T$ function value queries in total are needed to find the $\vareps$-KKT point $\vx^K$, where
	\begin{align*}
	\EE[T] &= \sum_{k=0}^{K-1} \EE[ T_k^{\mathrm{PPM}} T_k^{\mathrm{APG}} ] \\ 
	&= \tilde{O}\left(dK \vareps^{-3} (\log \vareps^{-1})^{q+2} \right)  =\tilde{O}\left(d\vareps^{-3} \right).
	\end{align*}
	
	In the special case when $\vc(\vx) = \vA \vx-\vb$, the term $\Vert \vc(\vx) \Vert^2 = \Vert \vA \vx-\vb \Vert^2$ is convex, so we have $\rho_c = 0$. Hence, by \eqref{eq:hat_rho_k}, $\hat{\rho}_k = \tilde{O}(1), \forall k \ge 0 $. Then following the same arguments as above, we obtain that for any $k \le K$,
	\begin{align*}
	\EE [ T_k^{\mathrm{PPM}} T_k^{\mathrm{APG}} ]&= O\left(\frac{\sqrt{\hat{L}_k \hat{\rho}_k}}{\vareps^{2}} (\log \vareps^{-1})^{q+2}\right)\\ 
	&= \tilde{O}\left(\vareps^{-\frac{5}{2}} \right).
	\end{align*}
	Therefore, at most $T$ total function value queries are needed to find the $\vareps$-KKT point $\vx^K$, where
	\begin{equation*}
	\EE[T] = \sum_{k=0}^{K-1} \EE[ T_k^{\mathrm{PPM}} T_k^{\mathrm{APG}} ]= \tilde{O}\left(\vareps^{-\frac{5}{2}} \right),
	\end{equation*}
	which completes the proof.

\section{Efficient Implementation of Algorithm \ref{alg:apcg}}
In this section, we provide Algorithm \ref{alg:apcg_eff}, which is a practical efficient implementation of the equivalent Algorithm \ref{alg:apcg}. 
Algorithm \ref{alg:apcg_eff} is efficient in the sense that it avoids the full-dimensional vector operations which exist in Algorithm \ref{alg:apcg}. Algorithm \ref{alg:apcg_eff} is equivalent to Algorithm \ref{alg:apcg} because their iterates satisfy the relations
\begin{align*}
    \vx^k &= \rho^k \vu^k + \vv^k, \\
    \vy^k &= \rho^{k+1} \vu^k + \vv^k, \\
    \vz^k &= -\rho^k \vu^k + \vv^k,
\end{align*}
which are proven in Proposition 1 of \cite{lin2014accelerated}.

\begin{algorithm}[h] 
	\caption{Efficient implementation of ZO-APCG for \eqref{eq:comp-prob}}\label{alg:apcg_eff}
	\DontPrintSemicolon
	\textbf{Input:} $\vx^{-1} \in \dom(\psi)$, tolerance $\vareps$, smoothness $L$, strong convexity $\mu$, and epoch length $l$. \;
	\textbf{Initialization:} $\vu^0 = \vzero, \vv^0 = \vx^0, \alpha = \frac{1}{d}\sqrt{\frac{\mu}{L}}, \rho = \frac{1-\alpha}{1+\alpha}$ \;
	\For{$k=0,1,\ldots, K-1$}{
		Sample $i_k \in [d]$ uniformly and compute
		$\tilde{\nabla}_{i_k} G(\vy^k), \st \Vert \tilde{\nabla}_{i_k} G(\vy^k) - \nabla_{i_k} G(\vy^k) \Vert \le E_{i_k}.$\;
		Compute
		$\vh^k_{i_k} = \argmin_{\vh \in \RR^{d_{i_k}}} \{ \frac{d\alpha L}{2} \Vert \vh \Vert^2 + \langle \tilde{\nabla}_{i_k} G(\rho^{k+1}\vu^k+\vv^k),\vh \rangle + H_{i_k}(-\rho^{k+1} \vu_{i_k}^k+\vv_{i_k}^k+\vh) \}$. \;
		$\vu^{k+1} = \vu^k, \vv^{k+1} = \vv^k$,
		$\vu^{k+1}_{i_k} = \vu^k_{i_k} - \frac{1-d\alpha}{2\rho^{k+1}}\vh^k_{i_k}, \vv^{k+1}_{i_k} = \vv^k_{i_k} + \frac{1+d\alpha}{2} \vh^k_{i_k}$.\;
		$\vx^{k+1} = \rho^{k+1}\vu^{k+1}+\vv^{k+1}$.\;
		\If{$k+1 \equiv 0 \pmod l$}{ 
		 Compute $\tilde{\nabla} G(\vx^{k+1}), \st \Vert \tilde{\nabla} G(\vx^{k+1}) - \nabla G(\vx^{k+1}) \Vert \le E$\; 
		 $\hat{\vx}^{k+1} = \argmin_{\vx \in \RR^d} \{ \langle \tilde{\nabla} G(\vx^{k+1}), \vx-\vx^{k+1} \rangle + \frac{L}{2} \Vert \vx-\vx^{k+1} \Vert^2 + H(\vx) \}$.\;
		 \textbf{Return} $\hat{\vx}^{k+1}$ and \textbf{stop} if $\dist (\vzero, \tilde{\partial} F(\hat{\vx}^{k+1})) \le \frac{3\vareps}{4}$.
		 }
	}
	
\end{algorithm}
\section{Additional Numerical Experiments}
In this section, we provide additional numerical experiments to demonstrate the empirical performance of the proposed ZO-iALM. All the tests were performed in MATLAB 2019b on a Macbook Pro with 4 cores and 16GB memory.

\subsection{Nonconvex Linearly Constrained Quadratic Programs (LCQP)}\label{subsec:qp}
In this subsection, we test the proposed method on solving nonconvex LCQP:
\begin{equation}\label{eq:ncQP} 
\begin{aligned}
&\textstyle \min_{\vx \in \mathbb{R}^n}  \frac{1}{2} \vx^\top \vQ \vx + \vc^\top  \vx,\\ 
&\text{s.t. } \vA\vx=\vb,\ x_i \in [l_i,u_i],\,\forall\, i\in [n],
\end{aligned}
\end{equation} \normalsize
where $\vA \in \mathbb{R}^{m\times n}$, and $\vQ\in\RR^{n\times n}$ is symmetric and indefinite (thus the objective is nonconvex). In the test, we generated all data randomly. The smallest eigenvalue of $\vQ$ is $-\rho < 0$, and thus the problem is $\rho$-weakly convex. For all tested instances, we set $l_i=-5$ and $u_i=5$ for each $i\in [n]$.

We generated an LCQP instance with $m = 10$, $n = 100$, and $\rho = 1$.
Since no other existing ZOMs are able to handle nonconvex constrained problems, we compare our proposed ZO-iALM to two other methods that replace our ZO-iPPM subroutine with ZO-AdaMM \cite{liu2018zeroth} and ZO-ProxSGD \cite{ghadimi2016mini} respectively. 
We set $\beta_k=\sigma^k\beta_0$ with $\sigma=3$ and $\beta_0=0.01$ for all iALM outer loops. In each method, we set $a = 10^{-4}$ to be the sampling radius. In ZO-AdaMM, we set $\alpha = 1, \beta_1 = 0.75, \beta_2 = 1$. In ZO-ProxSGD, we fix the step size to be $\frac{1}{nL}$, where $L$ is the smoothness constant of each subproblem.
The tolerance was set to $\vareps = 10^{-3}$ for the proposed ZO-iALM and $\vareps = 0.5$ for all compared methods since they could not converge with a tolerance as low as $0.001$. We also conducted experiments replacing our ZO-APCU inner solver with ZO-ARS in \cite{nesterov2017random}, but ZO-ARS in \cite{nesterov2017random} failed to converge.

In Table \ref{table:qp}, we report, for each method, the primal residual, dual residual, running time (in seconds), and the number of queries, shortened as \verb|pres|, \verb|dres|, \verb|time|, and \verb|#Obj|. 

From the results, we conclude that, to reach an $\vareps$-KKT point to the black-box LCQP problem, the proposed ZO-iALM needs significantly fewer queries to reach a significantly higher accuracy than all other compared methods.

\begin{table}[h]\caption{Results by the proposed ZO-iALM with ZO-iPPM, the ZO-AdaMM in \cite{chen2019zo}, and the ZO-ProxSGD in \cite{ghadimi2016mini} on solving a black-box $1$-weakly convex LCQP \eqref{eq:ncQP} of size $m=10$ and $n=100$. }\label{table:qp} 
\begin{center}
\resizebox{0.48 \textwidth}{!}{ 
\begin{tabular}{|c||cccc|cccc|cccc|} 
\hline
method & pres & dres & time & \#Obj
\\\hline\hline 
ZO-iALM & 9.61e-4 & 6.83e-4 & 11.42 & 2344400 \\ ZO-AdaMM & 0.42 & 0.47 & 61.86 & 8060000 \\
ZO-ProxSGD & 0.24 & 0.49 & 86.98 & 16520000 \\\hline 
\end{tabular}
}
\end{center}
\end{table}

\subsection{Unconstrained Strongly-convex Quadratic Programs (USCQP)}\label{subsec:uscqp}
In this subsection, we test the proposed core subsolver ZO-APCU on solving USCQP:
\begin{equation}\label{eq:uscqp} 
\begin{aligned}
&\textstyle \min_{\vx \in \mathbb{R}^n}  \frac{1}{2} \vx^\top \vQ \vx + \vc^\top  \vx,\\ 
\end{aligned}
\end{equation} \normalsize
where $\vQ\in\RR^{n\times n}$ is symmetric with $\mu > 0$ as its smallest eigenvalue (thus the objective is $\mu$ strongly-convex). In the test, we generated all data randomly. 

We generated an USCQP instance with $n = 100$, and $\mu = 1$. We compare our proposed ZO-APCU to ZO-AdaMM \cite{chen2019zo} and ZO-ARS \cite{nesterov2017random} respectively. In each method, we set $\vareps = 10^{-3}$ to be the error tolerance and $a = 10^{-5}$ to be the sampling radius. In ZO-AdaMM, we set $\alpha = 1, \beta_1 = 0.75, \beta_2 = 1$. In ZO-ProxSGD, we fix the step size to be $\frac{1}{nL}$, where $L$ is the smoothness constant of each subproblem. In ZO-ARS, we set $\theta = \frac{1}{16L(n+4)^2}, h = \frac{1}{4L(n+4)}, \alpha = \sqrt{L\theta}$, where $L$ is the smoothness constant of the objective.

In Figure \ref{fig:uscqp}, we compare the objective error trajectories of our method, ZO-AdaMM in \cite{chen2019zo}, and ZO-ARS in \cite{nesterov2017random}. For each method, we also report the objective error, gradient norm, running time (in seconds), and the query count, shortened as \verb|objErr|, \verb|normGrad|, \verb|time|, and \verb|#Obj| in Table \ref{table:uscqp}. From the results, we conclude that the proposed ZO-APCU reaches an $\vareps$-stationary point to the USCQP problem \eqref{eq:uscqp} with at least $3$ times fewer number of queries compared to all other methods. 

\begin{figure}[h]
	\begin{center}
			Unconstrained Strongly-convex Quadratic Programs \eqref{eq:uscqp} \\
	\includegraphics[width=0.3\textwidth]{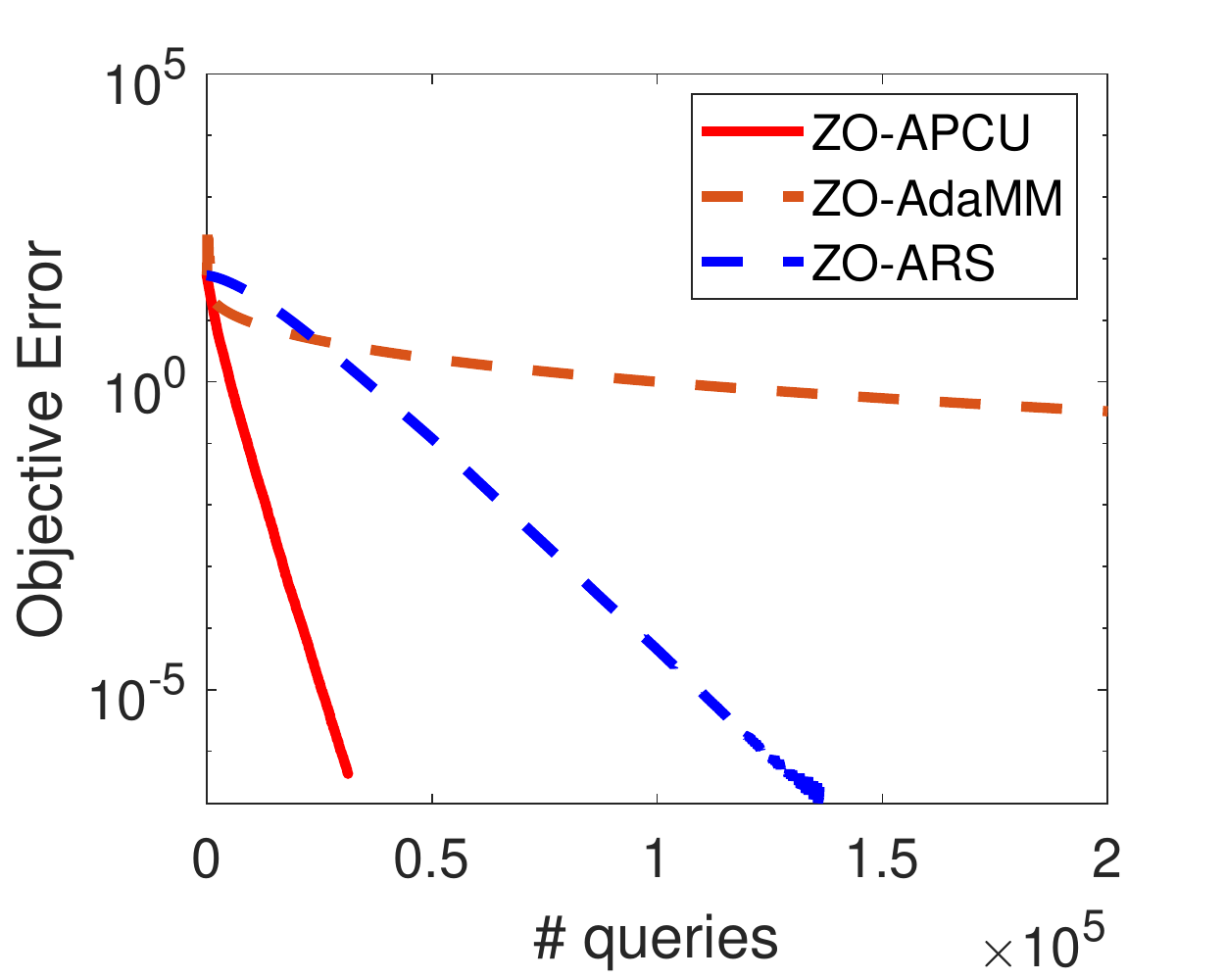}
	\end{center}
		\caption{Comparison of our proposed ZO-APCU, ZO-AdaMM in \cite{chen2019zo}, and ZO-ARS in \cite{nesterov2017random}. The plot shows the objective error.}\label{fig:uscqp}
\end{figure}

\begin{table}[h]\caption{Results by the proposed ZO-APCU, the ZO-AdaMM in \cite{chen2019zo}, and the ZO-ProxSGD in \cite{ghadimi2016mini} on solving the unconstrained strongly-convex quadratic programs \eqref{eq:uscqp}. }\label{table:uscqp} 
\begin{center}
\resizebox{0.48 \textwidth}{!}{ 
\begin{tabular}{|c||cccc|cccc|cccc|} 
\hline
method & objErr & normGrad & time & \#Obj
\\\hline\hline 
ZO-APCU & 4.29e-7 & 1.00e-3 & 0.50 & 31400 \\ ZO-AdaMM & 4.93e-7 & 9.99e-4 & 52.17 & 11576470 \\
ZO-ARS & 1.80e-7 & 4.96e-4 & 3.68 & 136200 \\\hline 
\end{tabular}
}
\end{center}
\end{table}

\subsection{Logistic Regression (LR)}\label{subsec:num_multi_pt}
In this subsection, we compare different multi-point coordinate gradient estimators proposed in Section \ref{sec:multi_pt} in the high accuracy setting. We use the proposed subsolver ZO-APCU on solving the logistic regression problem:
\begin{equation}\label{eq:LR} 
\min_{\vw,b} \frac{1}{N}\sum_{i=1}^N \log\big(1+\exp[-y_i(\vw^\top\vx_i+b)]\big) + \frac{\lambda}{2}\|\vw\|_2^2 + \frac{\lambda}{2}b^2, \\
\end{equation} \normalsize
where we are given the training data $\{(\vx_i,y_i)\}_{i=1}^N$ with $y_i\in\{+1,-1\}$ for each $i=1,\ldots,N$. Note that $\lambda$ is the strong convexity constant of the objective function. In the test, we use the \textit{spamdata} data set \cite{Dua:2019} with $N = 100$ as the number of randomly chosen data points and $n = 57$ as the variable dimension. In this subsection, we run two independent tests.

In the first test, we compare the final accuracy using our proposed ZO-APCU with 2-point and 4-point coordinate gradient estimators respectively. In each method, we set $\lambda = 1$ as the strong convexity constant, $\vareps = 10^{-11}$ as the error tolerance, $a = 10^{-5}$ as the sampling radius, and $K = 114000$ to be the maximum number of queries.

In Figures \ref{fig:LR_4pt_query} and \ref{fig:LR_4pt_iter}, we compare the gradient norm trajectories of the proposed ZO-APCU under 2-point and 4-point settings. For each setting, we also report the gradient norm, running time (in seconds), and the query count, shortened as \verb|normGrad|, \verb|time|, and \verb|#Obj| in Table \ref{table:LR_4pt}. From the results, we conclude that using the 4-point gradient estimator enables ZO-APCU to reach a more accurate solution to the LR problem \eqref{eq:LR} than using the 2-point gradient estimator.

\begin{figure}[h]
	\begin{center}
			Logistic Regression \eqref{eq:LR} \\
	\includegraphics[width=0.3\textwidth]{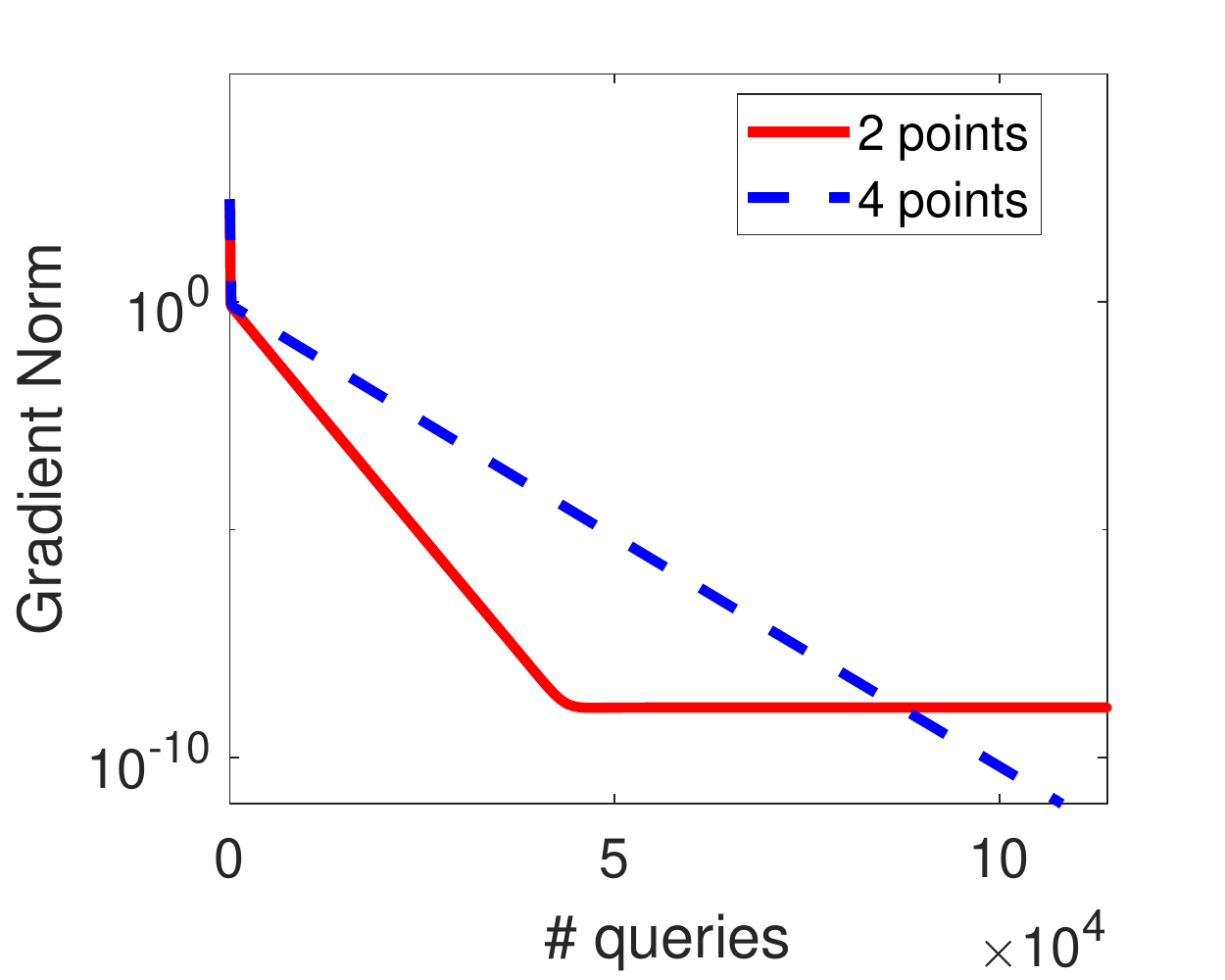}
	\end{center}
		\caption{Comparison of ZO-APCU with 2-point and 4-point gradient estimators. The plot shows the gradient norm versus the query count. The sampling radius is $a = 10^{-5}$.}\label{fig:LR_4pt_query}
\end{figure}
\begin{figure}[h]
	\begin{center}
			Logistic Regression \eqref{eq:LR} \\
	\includegraphics[width=0.3\textwidth]{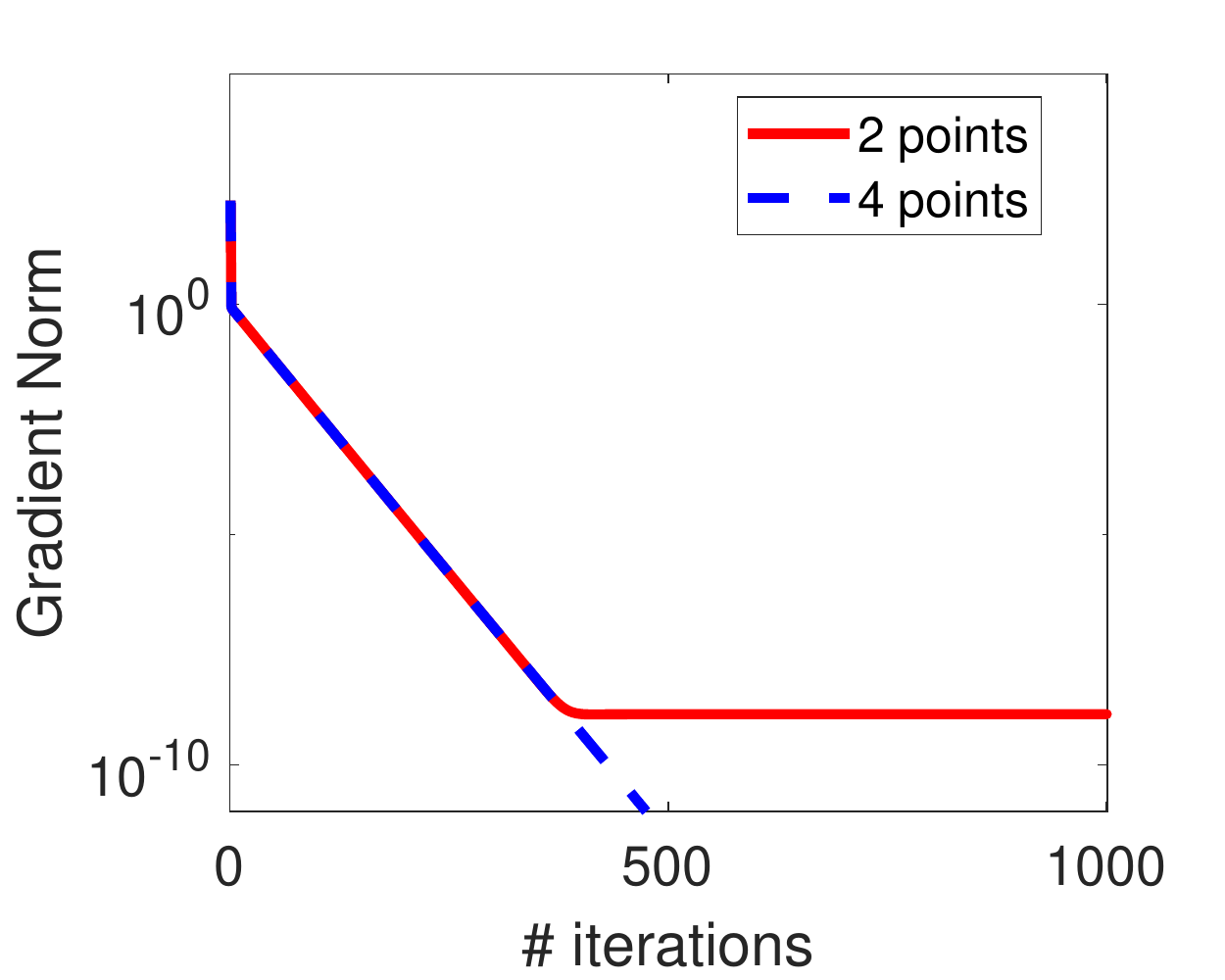}
	\end{center}
		\caption{Comparison of ZO-APCU with 2-point and 4-point gradient estimators. The plot shows the gradient norm versus the iteration count. The sampling radius is $a = 10^{-5}$.}\label{fig:LR_4pt_iter}
\end{figure}

\begin{table}[h]\caption{Results by the proposed ZO-APCU with 2-point and 4-point gradient estimators respectively on solving the logistic regression problem \eqref{eq:LR}. }\label{table:LR_4pt} 
\begin{center}
\begin{tabular}{|c||ccc|ccc|} 
\hline
\#Points & normGrad & time & \#Obj
\\\hline\hline 
2-pt & 1.26e-9 & 3.12 & 114000 \\ 4-pt & 9.68e-12 & 1.38 & 108072 \\\hline 
\end{tabular}
\end{center}
\end{table}

In the second test, we compare the final accuracy using our proposed ZO-APCU with 2-point, 4-point, and 6-point gradient estimators respectively, by a larger sampling radius $a = 10^{-2}$. In each method, we set $\lambda = 1$ as the strong convexity constant, $\vareps = 10^{-7}$ as the error tolerance, and $K = 114000$ to be the maximum number of queries.

In Figures \ref{fig:LR_6pt_query} and \ref{fig:LR_6pt_iter}, we compare the gradient norm trajectories of the proposed ZO-APCU under three settings. For each setting, we also report the gradient norm, running time (in seconds), and the query count, shortened as \verb|normGrad|, \verb|time|, and \verb|#Obj| in Table \ref{table:LR_6pt}. From the results, we conclude that when the sampling radius is large, to reach a decent accuracy to the LR problem \eqref{eq:LR}, it is beneficial to use more points in the gradient estimators.

\begin{figure}[t]
	\begin{center}
			Logistic Regression \eqref{eq:LR} \\
	\includegraphics[width=0.3\textwidth]{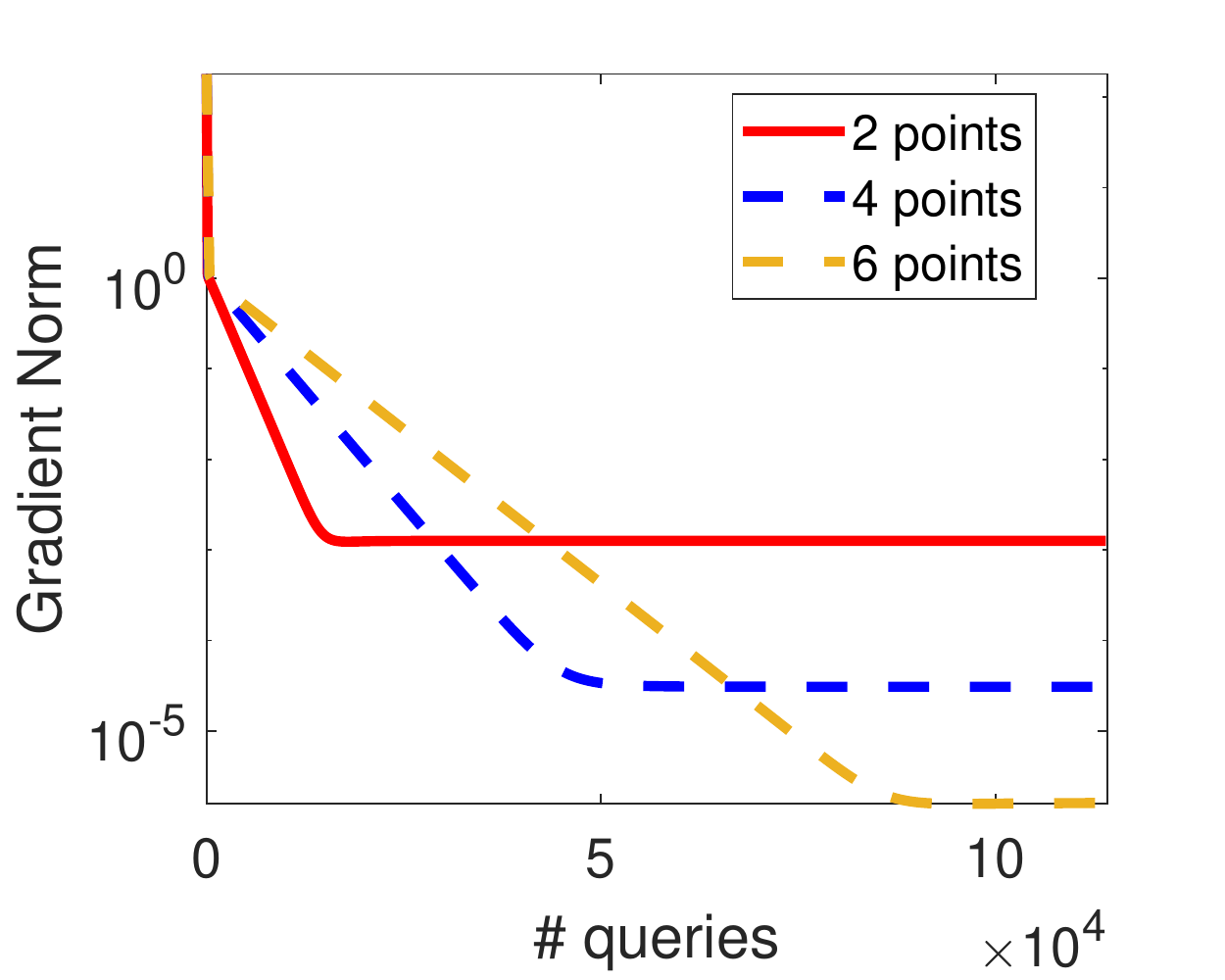}
	\end{center}
		\caption{Comparison of ZO-APCU with $(2,4,6)$-point gradient estimators. The plot shows the gradient norm versus the query count. The sampling radius is $a = 10^{-2}$. }\label{fig:LR_6pt_query}
\end{figure}
\begin{figure}[t]
	\begin{center}
			Logistic Regression \eqref{eq:LR} \\
	\includegraphics[width=0.3\textwidth]{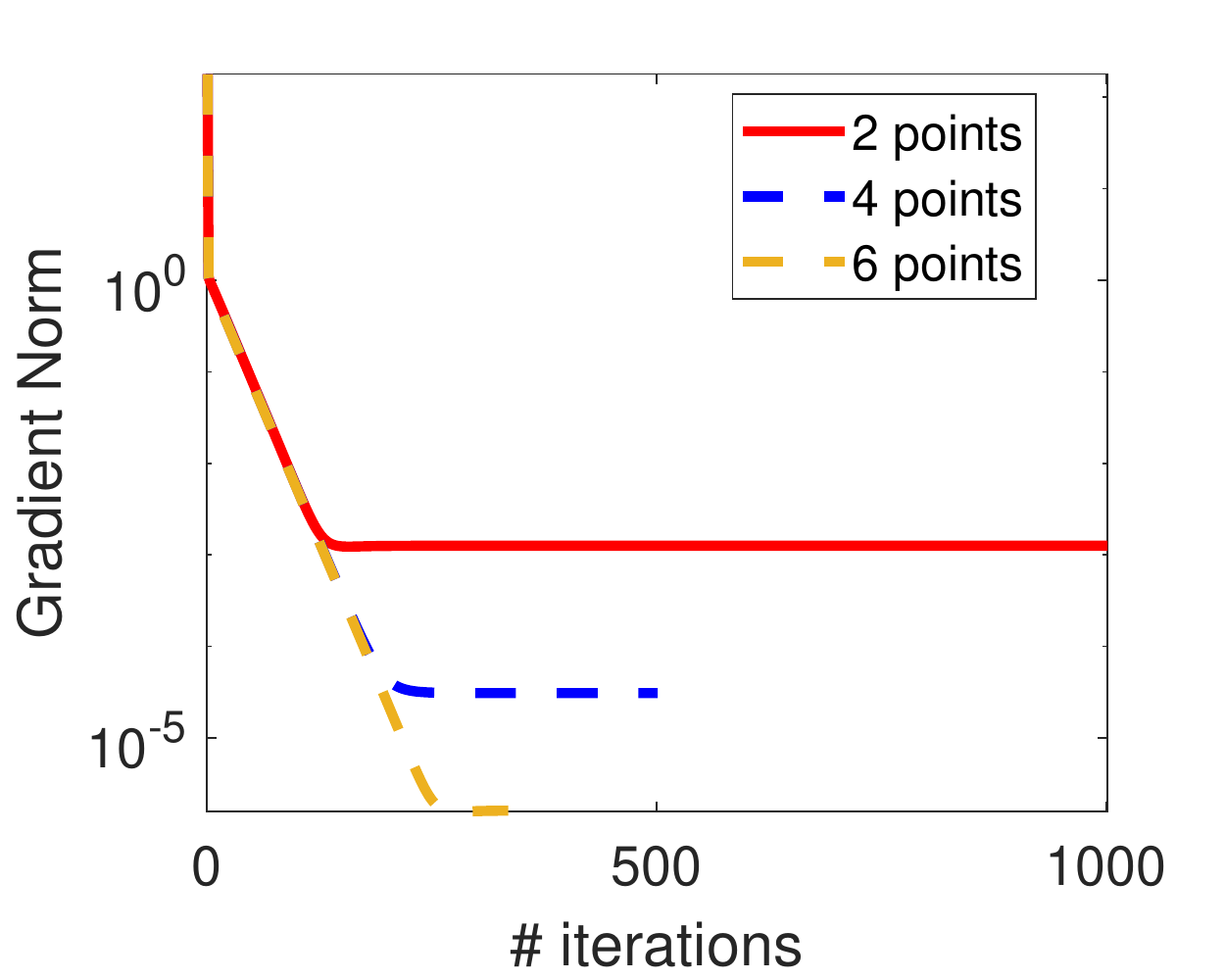}
	\end{center}
		\caption{Comparison of ZO-APCU with $(2,4,6)$-point gradient estimators. The plot shows the gradient norm versus the iteration count. The sampling radius is $a = 10^{-2}$.}\label{fig:LR_6pt_iter}
\end{figure}

\begin{table}[h]\caption{Results by the proposed ZO-APCU with $(2,4,6)$-point gradient estimators respectively on solving the logistic regression problem \eqref{eq:LR}. }\label{table:LR_6pt} 
\begin{center}
\begin{tabular}{|c||ccc|ccc|ccc|} 
\hline
\#Points & normGrad & time & \#Obj
\\\hline\hline 
2-pt & 1.3e-3 & 2.71 & 114000 \\ 4-pt & 3.08e-5 & 1.38 & 114000 \\ 6-pt & 1.60e-6 & 1.12 & 114000 \\\hline 
\end{tabular}
\end{center}
\end{table}

\section{Additional Table}
In Table~\ref{table:sensor}, for the resource allocation problem in sensor networks \eqref{eq:sensor_sel}, we report the primal residual, dual residual, running time (in seconds), and the query count, shortened as \verb|pres|, \verb|dres|, \verb|time|, and \verb|#Obj|.

\begin{table}[h]\caption{Results by the proposed ZO-iALM with ZO-iPPM, the ZO-AdaMM in \cite{chen2019zo}, and the ZO-ProxSGD in \cite{ghadimi2016mini} on solving the resource allocation problem in sensor networks \eqref{eq:sensor_sel}. }\label{table:sensor} 
\begin{center}
\begin{tabular}{|c||cccc|cccc|cccc|} 
\hline
method & pres & dres & time & \#Obj
\\\hline\hline 
ZO-iALM & 4.86e-2 & 7.01e-2 & 53.21 & 303790 \\ ZO-AdaMM & 2.11e-2 & 5.24e-2 & 80.09 & 659340 \\
ZO-ProxSGD & 4.97e-2 & 5.14e-2 & 427.38 & 3590277 \\\hline 
\end{tabular}
\end{center}
\end{table}

\end{document}